\documentclass[10pt]{amsart}
\usepackage{amsfonts}
\usepackage{amsmath}
\usepackage{amscd}
\usepackage{amssymb}
\usepackage{amsthm}
\usepackage{comment}
\usepackage{enumerate}
\usepackage{hyperref}
\usepackage{mathrsfs} 
\usepackage{breqn}
\usepackage{xcolor}
\usepackage{graphicx}
\usepackage{stmaryrd}
\usepackage{geometry}
\let\vec\mathbf
\newtheorem{thm}{Theorem}[section]

\newtheorem{prop}[thm]{Proposition}
\newtheorem{lem}[thm]{Lemma}
\newtheorem{cor}[thm]{Corollary}

\newtheorem{definition}[thm]{Definition}
\newtheorem{rmk}[thm]{Remark}

\newcommand{\CC}{\mathcal{C}}

\newcommand{\FF}{\mathfrak{F}}

\newcommand{\LL}{\mathcal{L}}
\newcommand{\PP}{\mathcal{P}}

\newcommand{\TN}{\mathcal{T}^N}
\newcommand{\mt}{\mathcal{T}^N}
\newcommand{\mf}{\mathfrak{F}}
\newcommand{\prst}{\mathbb{P}}

\newcommand{\rr}{\vec{r}_t}

\newcommand{\vc}[1]{{\bf #1}}
\newcommand{\T}{\mathcal{T}^N}
\newcommand{\torN}{\mathcal{T}^N}
\newcommand{\R}{\mathbb{R}}

\newcommand{\E}{\mathbb{E}}
\newcommand{\p}{\mathbb{P}}

\newcommand{\N}{\mathbb{N}}
\newcommand{\B}{\mathfrak{B}}
\newcommand{\BBB}{\mathscr{B}}

\newcommand{\Grad}{\nabla}
\usepackage{authblk}

\numberwithin{equation}{section}

\newcommand{\dd}{\mathrm{d}}
\newcommand{\dif}{\mathrm{d}}
\newcommand{\D}{\mathrm{d}}
\newcommand{\dx}{\,\mathrm{d}x}
\newcommand{\dt}{\,\mathrm{d}t}
\newcommand{\dxt}{\,\mathrm{d}x\,\mathrm{d}t}
\newcommand{\ds}{\,\mathrm{d}s}
\newcommand{\dxs}{\,\mathrm{d}x\,\mathrm{d} s}

\newcommand{\vr}{\varrho}

\newcommand{\bfxi}{\boldsymbol{\xi}}
\newcommand{\bfeta}{\boldsymbol{\eta}}
\newcommand{\bfsigma}{\boldsymbol{\sigma}}
\newcommand{\bfpsi}{\boldsymbol{\psi}}
\newcommand{\bfphi}{\boldsymbol{\varphi}}

\newcommand{\bfU}{\mathbf U}

\newcommand{\bff}{\mathbf f}

\newcommand{\bfq}{\mathbf q}
\newcommand{\bfu}{\mathbf u}

\newcommand{\bfm}{\mathbf m}

\newcommand{\Div}{\mathrm{div}}

\newcommand{\intTor}[1]{\int_{\torN} #1 \ \dx}

 \newcommand{\OTN}{\Omega}

\begin{document}

\title{Compressible Euler equations with transport noise}

\author{Richard Boadi}
\author{Dominic Breit}
\address{University of Duisburg--Essen, Faculty of Mathematics, The-Leymann-Stra\ss e 9, 45127 Essen, Germany.}
\address{TU Clausthal, Institute of Mathematics, Erzstra\ss e 1, 38678 Clausthal-Zellerfeld, Germany.}
\email{richard.boadi@uni-due.de}
\email{dominic.breit@uni-due.de}
\author{Thamsanqa Castern Moyo}
\address{Faculty of Mathematics and Computer Science, FernUniversit\"at in Hagen, Universit\"atsstra\ss e 47, 58097 Hagen, Germany}
\email{thamsanqacastern.moyo@fernuni-hagen.de}

\maketitle

\vspace{-3ex}
\begin{center}
\large{R. Boadi, D. Breit \& T. C. Moyo}
\end{center}
\normalsize

\begin{abstract}
We study the isentropic compressible Euler equations in multi-dimensions with stochastic perturbation of transport type. On the one hand, this is motivated by the physical modelling in turbulence theory. On the other hand, it has been shown recently that this type of noise can have regularising effects. In this paper, we prove  the existence of dissipative measure-valued martingale solutions, the weak-strong uniqueness property and  the existence of Markov selections.
\end{abstract}

%
%
%

\section{Introduction}
We consider the stochastic isentropic Euler equations, describing the flow of a compressible fluid in a bounded domain $\mathcal O \subset R^N$ in $N=2,3$ dimensions:
\begin{align}\label{1.1}
\left\{\begin{array}{rc}
\partial_t \bfm+\mathrm{div}\Big(\tfrac{\bfm\otimes\bfm}{\varrho}\Big)=-\nabla p+\dot{\bfeta},& \mbox{}\\
\partial_t\varrho+\Div (\bfm)=\dot{\xi}.\qquad\quad\quad\quad\,& \mbox{}\end{array}\right.
\end{align}
The unknowns are the momentum $\bfm$ and the density $\varrho$. For technical simplification we identify $\mathcal O$ with the flat torus $\mt$ and complement \eqref{1.1} with periodic boundary conditions. We suppose that the pressure $p$ is barotropic, that is, $p=p(\varrho)=\mathrm{Ma}^{-2}\varrho^\gamma$, where $\mathrm{Ma}>0$ is the Mach-number and $\gamma>1$ the adiabatic exponent.  The noise terms $\dot{\bfeta}$ and $\dot{\xi}$ can depend on the solution $(\bfm,\varrho)$ itself (note that $\dot{\xi}$ in \eqref{1.1}$_2$ is only physically meaningful if it does not perturb the conservation of mass, which requires $\int_{\mathcal O}\dot{\xi}\dx=0$). In the following we describe different scenarios concerning their role.\\\

\textbf{Deterministic case}. In the deterministic case (that is $\dot{\xi}=0$ and $\dot{\bfeta}=0$ or $\dot{\bfeta}=\varrho\bff$ with $\bff$ deterministic) it is well-known that unique smooth solutions only exist locally in time.
In the long-run, multiple solutions exist \cite{Smo} and solutions can develop singularities \cite{BSV1,BSV2}. Weak solutions (where derivatives only exist in the sense of distributions) can be constructed by the method of convex integration introduced in \cite{DS}, but they do not satisfy important physical principles such as energy conservation and maximal dissipation. The wider class of measure-valued solutions can overcome this drawback. Here the nonlinearities of the equations are described by measures capturing oscillations and concentrations of approximate solutions, see \cite{KrZa,Ne}. It is important to note that these solutions can be constructed to comply with the weak-strong uniqueness principle: a measure-valued solution coincides with the strong solution as soon as the latter exists, we refer the reader to  \cite{BDS,GSWW} for more details.\\

\textbf{Probabilistic case.} From a probabilistic point of view two concepts of solutions are common: probabilistically strong (pathwise solution) and Probabilistically weak (martingale solution). We speak of a probabilistically strong solution if the solution can be constructed with respect to a given noise on a given probability space.  We say a solution is  probabilistically weak if the underlying probability space is not a priori known,  but becomes an integral part of the solution. In what follows we distinguish between two different ways how stochasticity can enter the equations.

\textbf{Stochastic forcing}. We speak of stochastic forcing if
 \begin{align*}
\dot{\bfeta}=\Phi(\varrho,\bfm)\,\frac{\dd W}{\dt},\quad \dot{\xi}=0,
\end{align*}
with a (possibly infinite-dimensional) Wiener process $W$ and an operator $\Phi$ with appropriate growth assumptions. The easiest case is additive noise, where $\Phi(\varrho,\bfm)=\varrho\Psi$ with $\Psi$ independent of $\varrho$ and $\bfm$. The first results for \eqref{1.1} with stochastic forcing can be found in \cite{BeVo}, where the existence of martingale solutions to the one-dimensional problem has been shown. Local well-posedness (in smooth function spaces) in general dimensions is the object of \cite{BrMe,JUKim}.

In \cite{BFHci} global-in-time weak solutions to \eqref{1.1} have been constructed by the method of convex integration: In fact, there exist infinitely many weak solutions with the same initial data. However, these solutions have yet two drawbacks: They only exist for very smooth initial data and only up to a (possibly large) stopping time. On the other hand, it is quite interesting to note that these solutions are strong in the probabilistic sense (all of them exist on the same probability space). This is surprising since the existence of pathwise solutions is typically linked to some sort of uniqueness as a consequence of the classical Yamada-Watanabe theorem (existence of a martingale solution and pathwise uniqueness imply existence of a probabilistically strong solution). 

Measure-valued martingale solutions to \eqref{1.1} with stochastic forcing were constructed in \cite{HKS}. As in the deterministic case, they satisfy an energy inequality which gives rise to a weak-strong uniqueness principle (different from analytically weak solutions).

Much more is known in the viscous case: a systematic study of the compressible Navier--Stokes system with stochastic forcing has been initiated in \cite{BH} leading to the research monograph \cite{BFHbook}.\\\

\textbf{Transport noise}. We speak of transport noise if
\begin{align}\label{eq:transport}
\dot{\bfeta}=\sum_{k=1}^K\Div(\bfsigma_k\otimes\bfm)\circ\frac{\dd W_k}{\dt},\quad\dot{\xi}=\sum_{k=1}^K\Div(\varrho\bfsigma_k) \circ\frac{\dd W_k}{\dt}, 
\end{align}
with given vector fields $\bfsigma_k$ and stochastic differentials with respect to independent Wiener process $W_k$ in the Stratonovich sense (an infinite sum is possible too). 
The motivation for transport noise is twofold:
\begin{itemize}
\item Holm derived in \cite{HOLM2,HOLM,HOLM1} models with transport noise
in fluid dynamics with arguments from geometric mechanics with the aim of modelling turbulent effects (see, in particular, \cite{HOLM2} for the compressible case). The determination of  the vector fields $\bfsigma_{k}$, $k=1,\dots,K$ is discussed in \cite{Co1,Co2}.
\item Transport noise can have regularisation effects as demonstrated in \cite{FGP} for the transport equation and very recently in \cite{FL} for the 3D incompressible Navier--Stokes equations. 
\end{itemize}
The literature on \eqref{1.1} with transport noise is rare and the only contribution is \cite{ART}
giving local well-posedness.\footnote{To see that \eqref{1.1} fits into the abstract framework from \cite{ART} one has to rewrite \eqref{1.1}$_2$ in the velocity formulation by formerly dividing by $\varrho$.}  Let us also mention \cite{BFHM,BFHZ} where the compressible Navier--Stokes system (the viscous version of \eqref{1.1}) with transport noise is considered.\\\

The aim of the present paper is to initiate a systematic study of the compressible Euler systems \eqref{1.1} with transport noise  \eqref{eq:transport}.
The first step is to prove the existence of measure-valued martingale solutions satisfying an energy balance, see Theorem \ref{thm:main} for the complete statement. This is achieved as the vanishing viscosity limit of the corresponding Navier--Stokes system from \cite{BFHM}. The core of our analysis is
a refined stochastic compactness method employing Jakubowsi's version \cite{jakubow} of the Skorokhod representation theorem.
Interestingly, the limit procedure in the energy balance is easier than for the problem with stochastic forcing.
The reason for this is that the transport noise in \eqref{eq:transport} is energy conservative and hence we have a purely deterministic energy balance. A second step of our analysis is the weak strong-uniqueness result in Theorem \ref{thm_a}. It is motivated by \cite[Theorem 5.12]{FlaRom} and states that the probability distribution of a martingale solutions coincides with that of a strong solution up to the blow-up time of the latter.

For stochastic evolution problems, a further well-posedness result is of great interest:
the Markov property. It means that the probability distribution of the future only depends on the present state of the evolution and is independent of the past. This is rather obvious for problems with unique solutions, while one may hope to select solutions with the Markov if uniqueness is not available. This idea originates from Krylov \cite{Kr} who considered finite dimensional equations. Eventually, the existence of a Markov selection for the 3D incompressible Navier--Stokes equations was proved in \cite{FlaRom} which paved the way for other infinite dimensional problems. In particular, the compressible Navier--Stokes equations \cite{BFHAAP} as well as the 3D incompressible Euler equations \cite{HZZ} have been considered. Note that all these results are concerned with stochastic forcing. In Theorem \ref{Mselection} we obtain the first result for problems with transport noise. Again we benefit from the deterministic energy balance: Incorporating the energy balance with its stochastic terms is the most
delicate part in \cite{FlaRom} and subsequent papers. Our situation is significantly easier. 

\section{Mathematical framework}

\subsection{Stratonovich integration}
Let $(\Omega,\mathfrak F,(\mathfrak F_t)_{t\geq0},\mathbb P)$ be a filtered probability space and let $W=(W_k)_{k=1}^K$ be a collection of independent standard $(\mathfrak F_t)$-Wiener processes.
Set $\mathbb Q=(\bfsigma_k)_{k=1}^K$, where $\bfsigma_k:\TN\rightarrow\TN$, $k=1,\dots,K$ and $K\in\N,$ are smooth and solenoidal vector fields.
We define the Stratonovich integrals in \eqref{eq:transport} by means of the It\^{o}--Stratonovich correction.
First, we define the stochastic integrals
\begin{align*}
&\int_0^t\Div (\varrho \mathbb{Q} ) \,\D W=\sum_{k=1}^K\int_0^t\Div (\varrho \bfsigma_k ) \,\D W_k,\\
&\int_0^t\Div (\bfm\otimes \mathbb{Q} ) \,\D W=\sum_{k=1}^K\int_0^t\Div (\bfm\otimes \bfsigma_k ) \,\D W_k,
\end{align*}
as It\^{o}-integrals on the Hilbert spaces $W^{-\ell,2}(\TN)$ and $W^{-\ell,2}(\TN,R^N)$, $\ell>N/2+1$, respectively. Indeed, if $\varrho$ and $\bfm$
are $(\mathfrak F_t)$-adapted stochastic processes taking values in 
$C_{\rm{weak}}([0,T];L^{q}(\TN))$ and $C_{\rm{weak}}([0,T];L^{q}(\TN,R^N))$, $q > 1$, respectively, and the vector fields  $\bfsigma_{k}$ are bounded, then
the It\^{o}-integrals 
\begin{align*}
&\int_0^t\langle\Div (\varrho \mathbb{Q} ),\psi\rangle \,\D W=-\sum_{k=1}^K\int_0^t\int_{\TN}\varrho \bfsigma_k \cdot\nabla_x\psi\dx \,\D W_k,\quad\psi\in W^{\ell,2}(\TN),\\
&\int_0^t\langle\Div (\bfm\otimes \mathbb{Q} ),\bfpsi\rangle \,\D W=-\sum_{k=1}^K\int_0^t\int_{\TN}\bfm\otimes \bfsigma_k :\nabla_x\bfpsi\dx \,\D W_k,\quad\bfpsi\in W^{\ell,2}(\TN,R^N),
\end{align*}
are well-defined. The corresponding Stratonovich integrals are now defined via the It\^{o}--Stratonovich correction, that is
\begin{align*}
\int_0^t\int_{\TN}\varrho \bfsigma_k \cdot\nabla_x\psi\dx \,\circ\D W_k&=\int_0^t\int_{\TN}\varrho \bfsigma_k \cdot\nabla_x\psi\dx \,\D W_k\\&+\frac{1}{2}\Big\langle\Big\langle\int_{\TN}\varrho \bfsigma_k \cdot\nabla_x\psi\dx, W_k\Big\rangle\Big\rangle_t,\\
\int_0^t\int_{\TN}\bfm\otimes \bfsigma_k :\nabla_x\bfpsi\dx \,\circ\D W_k&=\int_0^t\int_{\TN}\bfm\otimes \bfsigma_k :\nabla_x\bfpsi\dx \,\D W_k\\&+\frac{1}{2}\Big\langle\Big\langle\int_{\TN}\bfm\otimes \bfsigma_k :\nabla_x\bfpsi\dx,W_k\Big\rangle\Big\rangle_t,
\end{align*}
where $\langle\langle\cdot,\cdot\rangle\rangle_t$ denotes the cross variation. We compute now the cross variations by means of \eqref{1.1} with \eqref{eq:transport}. Spelling out only the martingale part, we have
\begin{align*}
\left[\int_{\TN}\varrho \bfsigma_k \cdot\nabla_x\psi\dx\right](t)&=\dots-\sum_{\ell}\int_0^t\int_{\TN}\varrho \bfsigma_\ell \cdot\nabla_x(\bfsigma_k \cdot\nabla_x\psi)\dx \,\D W_\ell,\\
\left[\int_{\TN}\bfm\otimes \bfsigma_k :\nabla_{x}\bfpsi\dx\right](t)&=\dots-\sum_{\ell}\int_0^t\int_{\TN}\bfm\cdot (\bfsigma_\ell \cdot\nabla_{x}(\bfsigma_k\cdot\nabla_{x}\bfpsi))\dx \,\D W_\ell,
\end{align*}
where $\psi\in W^{\ell+1,2}(\TN)$ and $\bfpsi\in W^{\ell+1,2}(\TN,R^N)$.
Here the deterministic terms of the equations with quadratic variation zero are irrelevant and hidden. Plugging the previous considerations together we set
\begin{align*}
&\int_0^t\Div (\varrho \mathbb{Q} ) \,\circ\D W=\sum_{k=1}^K\int_0^t\Div (\varrho \bfsigma_k ) \,\D W_k+\frac{1}{2}\sum_{k=1}^K\int_0^t\Div(\bfsigma_k \otimes\bfsigma_k\nabla_{x}\varrho) \ds,\\
&\int_0^t\Div (\bfm\otimes \mathbb{Q} ) \,\circ\D W=\sum_{k=1}^K\int_0^t\Div (\bfm\otimes \bfsigma_k ) \,\D W_k+\frac{1}{2}\sum_{k=1}^K\int_0^t\Div(\bfsigma_k\otimes\bfsigma_k\nabla_{x}\bfm)\ds,
\end{align*}
to be understood in $W^{-\ell-1,2}(\TN)$ and $W^{-\ell-1,2}(\TN,R^N)$ respectively.

\subsection{Navier--Stokes equations}
We approximate \eqref{1.1} by its viscous counterpart
and pass eventually with the viscosity to zero. 
The system in question (with fixed viscosity coefficients) consists of the continuity equation for the fluid density $\varrho:\Omega\times[0,T]\times\TN\to R$ and the momentum equation for the  velocity field $\bfu:\Omega\times[0,T]\times\TN\to R^{N}$ given by
\begin{align}
\dd\varrho+\Div (\varrho\bfu)\dt&=\Div (\varrho \mathbb{Q} ) \circ\D W,\label{1.1a}\\
\dd(\varrho\bfu)+\Div\,\big(\varrho\bfu\otimes\bfu\big)\dt+\nabla  p(\varrho)\dt&=\Div\,\mathbb S(\nabla\bfu)\dt+\Div (\bfm\otimes \mathbb{Q} )\circ\D W,\label{1.1b}
\end{align}
 supplemented with the initial conditions $(\varrho(0),\varrho\bfu(0))=(\varrho_{0},\bfm_{0})$.
The stress tensor follows Newton's rheological law
\begin{align}\label{eq:nr}
\mathbb S(\nabla\bfu)=\mu\big(\nabla\bfu+\nabla\bfu^\top \big)+\lambda\Div\bfu\,\mathbb I
\end{align}
with strictly positive viscosity coefficients $\mu$ and $\lambda$. A solution to \eqref{1.1a}--\eqref{eq:nr} can be defined as follows.
\begin{definition}\label{MD1B}[Dissipative martingale solution]\label{MD1}
Let $(\varrho_0,\bfm_0)\in L^\gamma(\mt)\times L^{\frac{2\gamma}{\gamma+1}}(\mt)$ with $\varrho_0\geq0$ a.a. be deterministic initial data.
The quantity  $$\big((\Omega,\mf,(\mf_t)_{t\geq0},\prst),\varrho,\bfu,W)$$
is called a {\em dissipative martingale solution} to \eqref{1.1a}--\eqref{eq:nr} with initial data $(\varrho_0,\bfm_0)$, provided the following holds:
\begin{enumerate}[(a)]
\item $(\Omega,\mf,(\mf_t)_{t\geq 0},\prst)$ is a stochastic basis with a complete right-continuous filtration;
\item $W$ is a $K$-dimensional $(\mf_t)$-Wiener process;
%
    \item[(c)] The density $\varrho$ is $(\FF_t)$-adapted and satisfies $\p$-a.s.
    \[
    \varrho \in C_{\mathrm{loc}}([0,\infty), W^{-4,2}(\T))\cap L_{\mathrm{loc}}^{\infty}(0,\infty;L^{\gamma}(\T));
    \]
    \item[(d)] The momentum $\vec m:=\varrho\vec u$ is $(\FF_t)$-adapted and satisfies $\p$-a.s.
    \[
    \vec m\in C_{\mathrm{loc}}([0,\infty), W^{-4,2}(\T))\cap L_{\mathrm{loc}}^{\infty}(0,\infty;L^{\frac{2\gamma}{\gamma +1}}(\T));
    \]
\item The velocity field $\bfu$ is $(\FF_t)$-adapted\footnote{Adaptedness of the velocity field is understood in the sense of random distributions, cf. \cite[Chapter~2.2]{BFHbook}.} and satisfies $\p$-a.s. $$
 L^2_{\mathrm{loc}}(0,\infty; W^{1,2}(\mt,R^N));
$$
\item We have $(\varrho(0),(\varrho\bfu)(0))=(\varrho_0,\bfm_0)$ $\mathbb P$-a.s.;
\item The continuity equation holds in the sense that for all $t\geq0$
\begin{align}\label{eq:condW2}
\begin{aligned}
\int_{\mt}\varrho\psi\dx\bigg|_{s=0}^{s=t}&-
\int_0^t\int_{\mt}\varrho \bfu\cdot\nabla_{x}\psi\dxs\\&=
-\int_{\mt}\int_0^t\varrho \mathbb Q\cdot\nabla_x \psi\,\dd W\dx\\
&+\frac{1}{2}\sum_{k=1}^K\int_0^t\int_{\mt}\varrho\Div(\bfsigma_k \otimes\bfsigma_k\nabla_{x}\psi)\dx \ds
\end{aligned}
\end{align}
$\mathbb P$-a.s.  for all $\psi\in C^\infty(\TN)$;
\item The momentum equation holds in the sense that\begin{align}\label{eq:momdW2}
\begin{aligned}
\int_{\mt} \varrho \bfu\cdot \bfpsi\dx\bigg|_{s=0}^{s=t}& -\int_0^t\int_{\mt}\varrho \bfu\otimes \bfu:\nabla_x \bfpsi\dxs
\\
&+\int_0^t\int_{\mt}\mathbb S(\nabla_x \bfu):\nabla_x \bfpsi \dxs-\int_0^t\int_{\mt}
p(\varrho)\,\Div \bfpsi \dxs
\\&=-\int_{\mt}\int_0^t\varrho \bfu\otimes\mathbb Q:\nabla_x \bfpsi\,\dd W\dx\\&+\sum_{k=1}^K\int_0^t\int_{\mt}\varrho\bfu\cdot\Div(\bfsigma_k\otimes\bfsigma_k\nabla_{x}\bfpsi)\dx \ds
\end{aligned}
\end{align} 
$\mathbb P$-a.s. for all $\bfpsi\in C^\infty(\mt,R^{N})$;
\item The energy inequality is satisfied in the sense that
\begin{align} \label{eq:enedW2NS}
\begin{aligned}
- \int_0^\infty &\partial_t \psi \,
\mathscr E^{{\tt NS}} \dt+\int_0^\infty\psi\int_{\mt}\mathbb S(\nabla_x \bfu):\nabla_x \bfu\dxt\leq
\psi(0) \mathscr E^{{\tt NS}}(0)
\end{aligned}
\end{align}
holds $\mathbb P$-a.s. for any $\psi \in C^\infty_c([0, \infty))$.
Here, we abbreviated
$$\mathscr E^{{\tt NS}}(t)= \int_{\TN}\Big(\frac{1}{2} \varrho(t) | {\bfu}(t) |^2 + \frac{a}{\gamma-1}(\varrho(t))^\gamma\Big)\dx.$$
\end{enumerate}
\end{definition}
The following theorem is proved in \cite{BFHM}.\footnote{Note that an arbitrary finite time interval is considered in \cite{BFHM} but an extension to $(0,\infty)$ is straightforward.}

\begin{thm}\label{thm:main}
Let $\bfsigma_k \in C^\infty (\TN, R^N)$, ${\rm div}\bfsigma_k = 0$, $k=1,\dots,K$, and let $\gamma > \frac{N}{2}$, $N= 2,3$ be given.
Suppose that $\varrho_0\in L^\gamma(\mt),\bfm_0\in L^{\frac{2\gamma}{\gamma+1}}(\mt)$ with $\varrho_0\geq0$ a.a.  and $\varrho_0^{-1/2}\bfm_0\in L^2 (\mt, R^N)$ 
are deterministic initial data. 
Then there exists a dissipative martingale solution to \eqref{1.1a}--\eqref{eq:nr} with the initial condition $(\varrho_{0},\bfm_{0})$ in the sense of Definition \ref{MD1B}. 
\end{thm}

\subsection{Euler equations}\label{MVS}
In this section, we present solution concepts for the Euler equations
\begin{align}
\dd\varrho+\Div (\varrho\bfu)\dt&=\Div (\varrho \mathbb{Q} ) \circ\D W,\label{1.1A}\\
\dd\bfm+\Div\,\Big(\frac{\bfm\otimes\bfm}{\varrho}\Big)\dt+\nabla p( \varrho)\dt&=\Div (\bfm\otimes \mathbb{Q} )\circ\D W,\label{1.1B}
\end{align}
 supplemented with the initial conditions $(\varrho(0),\bfm(0))=(\varrho_{0},\bfm_{0})$.

\begin{definition}[Dissipative measure-valued martingale solution]\label{E:dfn}
Let $(\varrho_0,\bfm_0)\in L^\gamma(\mt)\times L^{\frac{2\gamma}{\gamma+1}}(\mt)$ with $\varrho_0\geq0$ a.a. be deterministic initial data. Then 
\[
((\Omega,\FF, (\FF_t)_{t\geq 0},\p),\varrho,\vec m,\mathcal{R}_{\tt{conv}},\mathcal{R}_{\tt{press}}, W)
\]
is called a dissipative measure-valued
martingale solution to (\ref{1.1A})--(\ref{1.1B}) with initial data $(\varrho_0,\bfm_0)$, provided\footnote{Some of our variables are not stochastic processes in the classical sense and we interpret their adaptedness in the sense of random distributions as introduced in \cite{BFHbook} (Chap. 2.2).}:

\begin{enumerate}
    \item [(a)] $(\Omega,\FF, (\FF_t)_{t\geq 0},\p)$ is a stochastic basis with complete right-continuous filtration;
    \item[(b)]$W$ is a $K$-dimensional $(\FF_t)$-Wiener process;
    \item[(c)] The density $\varrho$ is $(\FF_t)$-adapted and satisfies $\p$-a.s.
    \[
    \varrho \in C_{\mathrm{loc}}([0,\infty), W^{-4,2}(\T))\cap L_{\mathrm{loc}}^{\infty}(0,\infty;L^{\gamma}(\T));
    \]
    \item[(d)] The momentum $\vec m$ is $(\FF_t)_{t\geq 0}$-adapted and satisfies $\p$-a.s.
    \[
    \vec m\in C_{\mathrm{loc}}([0,\infty), W^{-4,2}(\T))\cap L_{\mathrm{loc}}^{\infty}(0,\infty;L^{\frac{2\gamma}{\gamma +1}}(\T));
    \]
    \item[(e)] The parametrised measures $(\mathcal{R}_{\tt{conv}},\mathcal{R}_{\tt{press}})$ are $(\FF_t)$-adapted and satisfy $\p$-a.s.
    \begin{align}\label{eq:para}
    t\mapsto \mathcal{R}_{\tt{conv}}(t) &\in L_{\mathrm{weak-(*)}}^{\infty}(0,\infty;\mathcal{M}^+(\T, R^{N\times N}));\\
    t\mapsto \mathcal{R}_{\tt{press}}(t) &\in L_{\mathrm{weak-(*)}}^{\infty}(0,\infty;\mathcal{M}^+(\T, R));
    \end{align}
    \item[(f)]  We have $\varrho(0,\cdot),\vec m (0,\cdot))=(\varrho_0,\vec m_0)$ $\mathbb P$-a.s.
    \item[(g)]
The continuity equation holds in the sense that for all $t\geq0$
    \begin{align}\label{eq:cont}
\begin{aligned}
       \left[\int_{\T}\varrho\varphi\, \dd x\right]_{s=0}^{s=t}& = \int_{0}^{t}\int_{\T} \vec m \cdot \nabla \varphi \, \dd x\,\dd s
-\int_{\mt}\int_0^t\varrho \mathbb Q\cdot\nabla_x \varphi\,\dd W\dx\\&+\sum_{k=1}^K\int_0^t\int_{\mt}\varrho\,\Div(\bfsigma_k\otimes\bfsigma_k\nabla_{x}\varphi)\dx \ds
\end{aligned}
    \end{align}$\p$-a.s.
for all $\varphi \in C^{\infty}(\T)$;
    \item[(h)]
The momentum equation holds in the sense that for all $t\geq0$
    \begin{align}\label{eq:mcxs}
\begin{aligned}
         \left[\int_{\T}\vec m \cdot \boldsymbol{\varphi}\dx\right]_{s=0}^{s=t}&=\int_{0}^{t}\int_{\T}\Big(\frac{\vec m \otimes\vec m}{\varrho}:\nabla\boldsymbol{\varphi}+p(\varrho)\Big)\, \dd x\, \dd s\\
         &+\int_{0}^{t}\int_{\T}\nabla \boldsymbol{\varphi}:\dd \mathcal{R}_{\tt{conv}}\, \dd s+\int_{0}^{t}\int_{\T}\mathrm{div} \boldsymbol{\varphi}\,\dd \mathcal{R}_{\tt{press}}\, \dd s\\
         &-\sum_{k=1}^K\int_{\mt}\int_0^t\bfm\otimes\bfsigma_k:\nabla_x \boldsymbol{\varphi}\,\dd W_k\dx\\&+\sum_{k=1}^K\int_0^t\int_{\mt}\bfm\cdot\Div(\bfsigma_k\otimes\bfsigma_k\nabla_{x}\boldsymbol{\varphi})\dx \ds
\end{aligned}
    \end{align}
$\p$-a.s.
for all $\boldsymbol{\varphi} \in C^{\infty}(\T)$;
    \item[(i)] 
The energy inequality is satisfied in the sense that
\begin{align} \label{eq:enedW2}
\begin{aligned}
- \int_0^\infty &\partial_t \psi \,
\mathscr E^{{\tt Euler}} \dt\leq
\psi(0) \mathscr E_0
\end{aligned}
\end{align}
holds $\mathbb P$-a.s. for any $\psi \in C^\infty_c([0, \infty))$.
Here, we abbreviated
$$\mathscr E^{\tt{Euler}} (t)= \int_{\TN}\Big(\frac{1}{2} \frac{| {\bfm}(t) |^2}{\varrho(t)} + \frac{a}{\gamma-1}(\varrho(t))^\gamma\Big)\dx+\frac{1}{2}\int_{\T}\dd \,\mathrm{tr}\mathcal{R}_{\tt{conv}}(t) +\frac{1}{\gamma-1}\int_{\T}\dd \,\mathrm{tr}\mathcal{R}_{\tt{press}}(t)$$
    for $t \geq 0$ and 
    \[
   \mathscr E_0= \int_{\T}\Big(\frac{1}{2}\frac{|\vec m_0 |^2}{\varrho_0}+\frac{a}{\gamma-1}\varrho_0^\gamma\Big)\, \dd x.
    \]
\end{enumerate}
\end{definition}
Our first main result is the following. 
\begin{thm}\label{thm:mainEuler}
Let $\bfsigma_k \in C^\infty (\TN, R^N)$, ${\rm div}\bfsigma_k = 0$, $k=1,\dots,K$, and let $\gamma > \frac{N}{2}$, $N= 2,3$ be given.
Suppose that $\varrho_0\in L^\gamma(\mt),\bfm_0\in L^{\frac{2\gamma}{\gamma+1}}(\mt)$ with $\varrho_0\geq0$ a.a.  and $\varrho_0^{-1/2}\bfm_0\in L^2 (\mt, R^N)$ 
are deterministic initial data. 
Then there exists a dissipative measure-valued martingale solution to \eqref{1.1A}--\eqref{1.1B} with the initial data $(\varrho_{0},\bfm_{0})$ in the sense of Definition \ref{E:dfn}. 
\end{thm}

The reason for the restriction $\gamma > \frac{N}{2}$ in Theorem \ref{thm:mainEuler} is purely technical and can be avoided. We prove Theorem \ref{thm:mainEuler} by approximating the system \eqref{1.1A}--\eqref{1.1B} by \eqref{1.1a}--\eqref{eq:nr} for physical reasons.
For the latter the assumption is crucial to obtain an analytically weak solution already in the deterministic case, cf. \cite{FeNoPe}. However, one may employ a different approximation to avoid it. As a by-product of our proof we obtain the following result concerning the inviscid limit of the compressible Navier--Stokes system with transport noise. 
\begin{cor}\label{cor:main}
Suppose that the assumptions from Theorem \ref{thm:mainEuler} are in place. Let $(\mu_n), (\lambda_n)$ be two strictly positive nullsequences.
If for any $n\in\N$
$$\big((\Omega^n,\mf^n,(\mf^n_t)_{t\geq0},\prst),\varrho^n,\bfu^n,W^n)$$ is dissipative martingale solution to \eqref{1.1a}--\eqref{eq:nr} with $\mu=\mu_n$ and $\lambda=\lambda_n$ in the sense of Definition \ref{MD1B} with initial data $(\varrho_0,\bfm_0)$, then there is a (non-relabelled) subsequence such that
\begin{align*}
        {\varrho}^{n} &\to \Tilde{\varrho} \quad \text{in}\quad C_{\mathrm{loc}}([0,\infty);W^{-4,2})\cap C_{\mathrm{weak}}([0,\infty); L^{\gamma}(\T)),\\
        \varrho^n{\vec{u}}^{n}&\to \Tilde{\vec m} \quad \text{in}\quad C_{\mathrm{loc}}([0,\infty);W^{-4,2}(\T))\cap C_{\mathrm{weak}} ([0,\infty);L^{\frac{2\gamma}{\gamma +1}}(\T)),
\end{align*}
in law, where $(\tilde\varrho,\tilde\bfm)$ is a dissipative solution to \eqref{1.1A}--\eqref{1.1B} in the sense of Definition \ref{E:dfn} with initial data $(\varrho_0,\bfm_0)$.
\end{cor}

In the following, we will give the definition of a local strong pathwise solution. These solutions are strong in the probabilistic and PDE sense and exist locally in time. In particular, the Euler system (\ref{1.1A})--\eqref{1.1B} will
be satisfied pointwise in space-time on the given stochastic basis associated
to the Wiener process $W$.

\begin{definition}[Strong solution]\label{def:compstrong}
Let $(\Omega, \FF, (\FF_t)_{t\geq 0}, \p)$ be a complete stochastic basis with a right-continuous filtration, let $W$ be a $K$-dimensional $(\FF_t)$-Wiener process and $\ell>\frac{N}{2}+1$. The tuple $[r, \vec v]$ and a stopping time $\mathfrak{s}$ is called a (local) strong pathwise solution to the system (\ref{1.1A})--\eqref{1.1B} provided:

\begin{itemize}
    \item[(a)] the density $r> 0$ $\p$-a.s., $t \mapsto r (t\wedge \mathfrak{s},\cdot)\in W^{\ell,2}(\T)$ is $(\FF_t)$-adapted,
    \[
    \E\left[\sup_{t\in [0,T]}\|r (t\wedge \mathfrak{s},\cdot)\|_{W^{\ell,2}(\T)}^q\right]<\infty \quad\text{for all $1\leq q <\infty$, $T>0$},;
    \]
    \item[(b)] the velocity $t\mapsto \vec v(t\wedge \mathfrak{s},\cdot) \in W^{\ell,2}(\T)$is $(\FF_t)$-adapted,
    \[
    \E\left[\sup_{t\in [0,T]}\|\vec v (t\wedge \mathfrak{s},\cdot)\|_{W^{\ell,2}(\T)}^q\right]<\infty \quad\text{for all $1\leq q <\infty$, $T>0$};
    \]
    \item[(c)]  for all $t\geq0$ there holds $\p$-a.s.
    \begin{align*}
    r(t\wedge\mathfrak{s}) &= \varrho(0)- \int_{0}^{t\wedge \mathfrak{s}}\mathrm{div}_x(r\vec v)\,\dd t+\int_{0}^{t\wedge \mathfrak{s}}\Div(r\mathbb Q)\, \dd W\\&+\sum_{k=1}^K\int_0^{t\wedge\mathfrak s}\Div(\bfsigma_k\otimes\bfsigma_k\nabla_{x}r)\dt,\\
    (r\vec v)(t\wedge \mathfrak{s})&=(r\vec v)(0)-\int_{0}^{t\wedge\mathfrak s}\mathrm{div}(r\vec v\otimes \vec v)\, \dd t-\int_{0}^{t\wedge\mathfrak{s}}\nabla_x p(r)\,\dd t \\&+\int_{0}^{t\wedge \mathfrak{s}}\Div(r\vec v\otimes\mathbb Q)\, \dd W+\sum_{k=1}^K\int_0^{t\wedge\mathfrak s}\Div(\bfsigma_k\otimes\bfsigma_k\nabla_{x}(r\vec v))\dt.
    \end{align*}
\end{itemize}
\end{definition}
Existence of a solution (\ref{1.1A})--\eqref{1.1B} in the sense of Definition \eqref{def:compstrong} is proved in \cite {ART}.


\section{Existence of solutions}\label{mvsproof}

The aim of this section is to prove Theorem \ref{thm:mainEuler}. Let $(\mu_n), (\lambda_n)$ be two strictly positive nullsequences. Suppose that $\varrho_0\in L^\gamma(\mt),\bfm_0\in L^{\frac{2\gamma}{\gamma+1}}(\mt)$ with $\varrho_0\geq0$ a.a.  and $\varrho_0^{-1/2}\bfm_0\in L^2 (\mt, R^N)$ 
are deterministic initial data. 
By Theorem \ref{thm:main} there is for any $n\in\N$ a multiplet
$$\big((\Omega^n,\mf^n,(\mf^n_t)_{t\geq0},\prst),\varrho^n,\bfu^n,W^n)$$
which is dissipative martingale solution to \eqref{1.1a}--\eqref{eq:nr} with $\mu=\mu_n$ and $\lambda=\lambda_n$ in the sense of Definition \ref{MD1B} with initial data $(\varrho_0,\bfm_0)$. Without loss of generality we can assume that the probability space $(\Omega^n,\mf^n,\prst)$ is independent of $n$ (in fact, one can choose the standard probability space $([0,1],\overline{\mathscr B([0,1])},\mathcal L^1)$).

First of all, by the energy inequality \eqref{eq:enedW2} we obtain 
\begin{align*}
\sup_{t\geq0}\int_{\TN}\Big(\frac{1}{2} \frac{| {\bfm^n}(t) |^2}{\varrho^n(t)} + \frac{a}{\gamma-1}(\varrho^n(t))^\gamma\Big)\dx+\int_0^\infty\int_{\TN}\mathbb S_n(\nabla\bfu^n):\nabla\bfu^n\dxt\lesssim \mathscr E_0
\end{align*}
uniformly in the sample space and uniformly in $n$.Here $\mathbb S_n$ is computed via \eqref{eq:nr} with $\mu=\mu_n$ and $\lambda=\lambda_n$.
Hence
we can deduce the following bounds
\begin{align}\label{eq:2512}\begin{aligned}
    \varrho^n&\in L^{\infty}(0,\infty;L^{\gamma}(\T)), \\
    \vec m^n &\in L^{\infty}(0,\infty;L^{\frac{2\gamma}{\gamma +1}}(\T)),\\
    \frac{\vec m^n}{\sqrt{\varrho^n}}&\in L^{\infty}(0,\infty;L^{2}(\T)),\\
    \frac{\vec m^n \otimes \vec m^n}{\varrho^n}&\in L^{\infty}(0,\infty;L^{1}(\T)),\\
\mathbb S_n(\nabla\bfu^n):\nabla\bfu^n&\in L^1(0,\infty;L^1(\mt)),
\end{aligned}
\end{align}
uniformly in the sample space and uniformly in $n$.

 In order to prepare for the proof of compactness we need some time-regularity of the variables $\varrho^n$ and $\vec m^n$.
 For the stochastic term in the momentum equation we have for $T>0$ arbitrary
\begin{align*}
\E \left[ \left \|\int_{0}^{\cdot} \Div(\bfm^n\otimes\mathbb Q) \,\dd W\right\|_{C^{\alpha}([0,T];W^{-2,2}(\T))}^{q}\right] \lesssim\,\max_{k\in\{1,\dots, K\}}\E \left[ \int_{0}^{T}\big\|\bfm^n\otimes \bfsigma_k\big\|^{q}_{W^{-1,2}(\mt)} \, \dd t \right]  
\end{align*}
where $q>2$ and $\alpha\in(0,1/2-1/q)$, cf.~\cite[ Lemma 4.6]{Hofm}. Now, by Sobolev's embedding and $\gamma>\frac{N}{2}$
\begin{align*}
\E \left[ \left \|\int_{0}^{\cdot} \Div(\bfm^n\otimes\mathbb Q) \,\dd W\right\|_{C^{\alpha}([0,T];W^{-k,2}(\T))}^{q}\right] \lesssim\,\E \left[ \int_{0}^{T}\big\|\bfm^n\big\|^q_{L^{\frac{2\gamma}{\gamma+1}}(\mt)} \, \dd t \right],
\end{align*}
which is uniformly bounded by \eqref{eq:2512}.
The deterministic terms in the momentum equation are bounded in $W^{1,2}(0,\infty;W^{-3,2}(\mt))$. Using the embeddings $W^{1,2}_t \hookrightarrow C^{\alpha}_{t}$ and $W^{-2,2}_x \hookrightarrow W^{-3,2}_x$ shows 
 \begin{equation}\label{estimate}
\E \left[\big\|\vec{m}^n\big\|_{C^{\alpha}([0,T];W^{-3,2}(\T))} \right] \leq c(T)
 \end{equation}
for all $\alpha < 1/2$. An analogous chain of arguments applied to the continuity equation yields
\begin{align}\label{estimate'}
\E\left[\big\| \varrho^n\big\|_{C^{\alpha}([0,T];W^{-3,2}(\T))}\right] \leq c(T).
\end{align}

\subsection{Compactness Argument}

%
%

We set the spaces:

\begin{align*}
\mathscr{X}_{\Vec{m}}&: = C_{\mathrm{loc}}(0,\infty;W^{-4,2}(\T))\cap L^\infty_{{\mathrm{weak}-(*)}} (0,\infty;L^{\frac{2\gamma}{\gamma +1}}(\T)), \\
\mathscr{X}_{\varrho}&:=C_{\mathrm{loc}}(0,\infty;W^{-4,2}(\T))\cap L^\infty_{\mathrm{weak}-(*)}(0,\infty; L^{\gamma}(\T)),\\  \mathscr{X}_{\text{C}}&:=L^{\infty}_{\mathrm{weak}-(*)}(0,\infty;\mathcal{M}(\T,R^{N\times N})),\quad
\mathscr{X}_{\text{P}}:=L^{\infty}_{\mathrm{weak}-(*)}(0,\infty;\mathcal{M}^{+}(\T)),\\\mathscr X_{\vec U}&:=L^{2}(0,\infty;L^2(\mt))\quad \mathscr{X}_{W} :=C_{\mathrm{loc}}(0,\infty;R^K),
\end{align*}
 with respect to weak-* topology for all spaces with $L^{\infty}(\cdot,\mathcal{M}^{\cdot}(\cdot))$. Furthermore, we choose the product path space
\begin{equation}
  \mathfrak{X}:=\mathscr{X}_{\varrho }\times \mathscr{X}_{ \Vec{m}}\times\mathscr X_{\vec U} \times\mathscr X_{\text{C}}\times \mathscr X_{\text{P}}\times \mathscr{X}_{W},
\end{equation}
%
where $\mathscr{X}_{\text{C}}$ and $\mathscr{X}_{\text{P}}$ are the path spaces for the nonlinear terms
\[
 \mathrm{P}^h:=a(\varrho^n)^{\gamma} ,\qquad\mathrm{C}^n:=\frac{\vec m^n\otimes \vec m^n}{\varrho^n},
\]
respectively. 
We denote by $\mathscr B_{\mathfrak X}$ the Borelian $\sigma$-algebra on $\mathfrak X$ and by $$\mathcal{L}[\varrho^n,\vec{m}^n,\mathbb S_n(\nabla \vec u^n),\mathrm P^n, \mathrm C^n, W^n]$$ the probability law on $\mathfrak{X}$. 
By \eqref{eq:2512}--\eqref{estimate'} we can infer that it is a sequence of tight measures on $\mathfrak{X}$.
In view of Jakubowksi's version of the Skorokhod representation theorem \cite{Jaku} (see also \cite{BZ}), we have the following proposition after passing to a non-relabelled subsequence.

\begin{prop} \label{skorokhod}
There exists a complete probability space  $(\Tilde{\Omega}, \Tilde{\FF},\Tilde{\p})$ with $(\mathfrak{X},\mathscr{B}_{\mathfrak{X}})$-valued random variables  $(\Tilde{\varrho}^{n},\Tilde{\vec{m}}^{n}, \tilde{\vec U}^n\Tilde{\mathrm P}^{n}, \Tilde{\mathrm C}^{n} ,\Tilde{W}^{n})$, $n \in \N$,\\ and $(\Tilde{\varrho},\Tilde{\vec{m}}, \Tilde{\mathrm P},\Tilde{\mathrm C}, \Tilde{W})$ such that
\begin{itemize}
    \item [(a)] For all $n \in \N$ the law of $$(\Tilde{\varrho}^{n},\Tilde{\vec{m}}^{n},\tilde{\vec U}^n, \Tilde{\mathrm P}^{n}, \Tilde{\mathrm C}^{n},\Tilde{W}^{n})$$ on $\mathfrak{X}$ coincides with $$\mathcal{L}[\varrho^{n},\vec{m}^{n}, \mathbb S_n(\nabla\bfu^n),\mathrm P^{n}, \mathrm C^{n}, W^n];$$ 
    \item[(b)] The sequence 
$$(\Tilde{\varrho}^{n},\Tilde{\vec{m}}^{n},\Tilde{\vec U}^n,\Tilde{\mathrm P}^{n}, \Tilde{\mathrm C}^{n},\Tilde{W}^{n}), \,\,n \in \N,$$ converges $\Tilde{\p}$-almost surely to  $$(\Tilde{\varrho},\Tilde{\vec{m}}, 0,\Tilde{\mathrm P},\Tilde{\mathrm C},\Tilde{W})$$ in the topology of $\mathfrak{X}$, i.e.
    \begin{equation}
        \begin{cases}
        \Tilde{\varrho}^{n} \to \Tilde{\varrho} \,\, \text{in }\,\, C_{\mathrm{loc}}(0,\infty;W^{-4,2}(\T))\cap L^\infty_{\mathrm{weak}-(*)}(0,\infty; L^{\gamma}(\T)),\\
       \Tilde{\vec{m}}^{n} \to \Tilde{\vec m} \, \, \text{in}\, \,  C_{\mathrm{loc}}(0,\infty;W^{-4,2}(\T))\cap L^\infty_{{\mathrm{weak}-(*)}} (0,\infty;L^{\frac{2\gamma}{\gamma +1}}(\T)) ,\\
   \Tilde{\vec{U}}^{n} \to 0 \, \, \text{in}\, \,  L^2 (0,\infty;L^{2}(\T)) ,\\
        \Tilde{\mathrm C}^{n} \to \overline{\Tilde{\mathrm C}} \,\, \text{in} \,\, L_{\mathrm{weak}-(*)}^{\infty}(0,\infty;\mathcal{M}^{+}(\T,R^{N\times N})), \\
        \Tilde{\mathrm P}^{n} \to \overline{\Tilde{ \mathrm P}} \,\, \text{in} \,\, L_{\mathrm{weak}-(*)}^{\infty}(0,\infty;\mathcal{M}^{+}(\T)), \\
         \Tilde{W}^{n} \to \Tilde{W} \,\, \text{in} \,\, C_{\mathrm{loc}}(0,\infty;R^K),
        \end{cases}
    \end{equation}
    $\Tilde{\p}$-a.s.
\end{itemize}
\end{prop}

To guarantee adaptedness of random variables and to ensure that the stochastic integrals are well-defined on the new probability space we introduce filtrations for correct measurability.
Let $\tilde{\FF}_t$ and $\tilde{\FF}_{t}^{n}$ be the $\tilde{\p}$-augmented filtration of the correspnding random variables from Proposition \ref{skorokhod}, i.e.,
\begin{align*}
\tilde{\FF}_t &=\sigma\Big( \sigma (\sigma(\rr \Tilde{\varrho},\rr\Tilde{\bfm},\rr\tilde{W}) \cup \sigma_t( \tilde{P},\tilde{C})\cup\{ \mathcal{N} \in \tilde{\FF};\tilde{\p}(\mathcal{N})=0\})\Big),\quad t\geq 0,\\
\tilde{\FF}_{t}^{n}&=\sigma\Big(\sigma(\rr\Tilde{\varrho}^{n},\rr\Tilde{\bfm}^n,\rr\tilde{W}^{n})\cup \sigma_t (\tilde{P}^{n},\tilde{C}^{n})\cup\{ \mathcal{N} \in \tilde{\FF};\tilde{\p}(\mathcal{N})=0\})\Big),\quad t\geq 0.
\end{align*}

Here $\rr$ denotes the restriction operator to the interval $[0,t]$ on the path space and $\sigma_t$ denotes the history of a random distribution.\footnote{The family of $\sigma$-fields $(\sigma_{t}[\vec V])_{t\geq 0}$ given as random distribution history of 
 \begin{equation*}
     \sigma_t[\vec V]:= \bigcap_{s>t}\sigma\left(\bigcup_{\varphi\in C_c^{\infty}(Q;R^N)}\{\langle \vec V, \varphi \rangle <1 \}\cup \{\mathcal N\in \FF, \p(\mathcal N)=0\} \right)
 \end{equation*}
 is called the history of $\vec V$. In fact, any random distribution is adapted to its history, see \cite[Chap. 2.2]{BFHbook}.}

\subsection{The new probability space}
In this section we will use the elementary method from \cite{BZ} to show that the approximated equations hold in the new probability space. The essence of this elementary method is to identify the quadratic and cross variations corresponding to the martingale with limit Wiener process obtained via compactness. Now in view of Proposition {\ref{skorokhod}}, we note that $\tilde{W}$ has the same law as $W$.  For all $t \in [0,T]$ and $\varphi \in C_{c}^{\infty}(\T) $ define the functionals:

\begin{align*}
\mathscr{M}( \varrho,\vec m,\vec U, C,P)_t &=\int_{\T}(\vec{m} -\vec{m}_{0}) \cdot  \boldsymbol{\varphi}  \, \dd x - \int_{0}^{t}\int_{\T}\nabla  \boldsymbol{\varphi} \, \dd C \dd s\\
&  \int_{0}^{t}\int_{\T} \vec{U}:\nabla  \boldsymbol{\varphi}  \, \dd x \dd s -\int_{0}^{t}\int_{\T} \mathrm{div}  \boldsymbol{\varphi}  \, \dd P \dd s\\
&+\sum_{k=1}^K\int_0^t\int_{\mt}\bfm\cdot\Div(\bfsigma_k\otimes\bfsigma_k\nabla_{x}\boldsymbol{\varphi})\dx \,\dd s,\\
\Psi(\bfm)_t &=\sum_{k=1}^K\int_{0}^{t}\left(\int_{\T}{\bfm}\otimes\mathbb Q:\nabla\boldsymbol{\varphi} \ \text{d}x\right)^2 \ \text{d}s,\\
\Psi_k(\bfm)_t &=-\int_{0}^{t}\int_{\T}\bfm\otimes\mathbb Q:\nabla\boldsymbol{\varphi}  \ \text{d}x \text{d}s.\\
\end{align*}

Now, we define the increment $$\mathscr{M}(\tilde{\varrho}_{n},\tilde{\vec m}_{n},\tilde{\vec U}_{n}, \tilde{C}_{n},\tilde{P}_{n})_{s,t}:=\mathscr{M}( \tilde{\varrho}_{n},\tilde{\vec m}_{n},\tilde{\vec U}_{n}, \tilde{C}_{n},\tilde{P}_{n})_{t}-\mathscr{M}( \tilde{\varrho}_{n},\tilde{\vec m}_{n},\tilde{\vec U}_{n}, \tilde{C}_{n},\tilde{P}_{n})_{s}$$ and similarly for $\Psi(\bfm)_{s,t}$ and 
$\Psi_k(\bfm)_{s,t}$. On the new probability space, completeness of the proof follows from showing that 

\begin{equation}
\mathscr{M}(\tilde{\varrho}_{n},\tilde{\vec m}_{n},\tilde{\vec U}_{n}, \tilde{C}_{n},\tilde{P}_{n})_{t} = -\int_{\TN}\int_{0}^{t}\tilde{\bfm}_{n}\otimes\mathbb Q:\nabla\boldsymbol{\varphi} \,\mathrm{d}\tilde{W}^{n}\dx.
\label{DS}
\end{equation}

For (\ref{DS}) to hold, it  suffices  to show that $\mathscr{M}(\tilde{\varrho}_{n},\tilde{\vec m}_{n},\tilde{\vec U}_{n}, \tilde{C}_{n},\tilde{P}_{n})_{t}$ is an  $(\tilde{\FF}_t^{n})$-martingale process and its  corresponding quadratic and cross variations satisfy, respectively,

\begin{equation}\label{variations}
 \bigg \langle \bigg \langle \mathscr{M}( \tilde{\varrho}_{n},\tilde{\vec m}_{n},\tilde{\vec U}_{n}, \tilde{C}_{n},\tilde{P}_{n}) \bigg \rangle \bigg\rangle =\Psi(\bfm),
\end{equation}

\begin{equation}\label{variations_1}
    \bigg\langle \bigg\langle \mathscr{M}( \tilde{\varrho}_{n},\tilde{\vec m}_{n},\tilde{\vec U}_{n}, \tilde{C}_{n},\tilde{P}_{n}), \tilde{W}^n_k\bigg \rangle \bigg \rangle = \Psi_k(\bfm),
\end{equation}

and consequently

\begin{equation}\label{cross_var}
\bigg \langle \bigg \langle \mathscr{M}( \tilde{\varrho}_{n},\tilde{\vec m}_{n},\tilde{\vec U}_{n}, \tilde{C}_{n},\tilde{P}_{n}) +\int_{\TN}\int_{0}^{\cdot}\tilde{\bfm}_{n}\otimes\mathbb Q:\nabla\boldsymbol{\varphi} \,\mathrm{d}\tilde{W}^{n}\dx\bigg \rangle \bigg\rangle = 0,    
\end{equation}
which implies the desired equation on the new probability space. We note that (\ref{variations}) and (\ref{variations_1}) hold based on the following observation: the mapping
\[
({\varrho},{\vec m},{\vec U}, {C},{P}) \mapsto \mathscr{M}( {\varrho},{\vec m},{\vec U}, {C},{P})_t
\]
is {well-defined} and continuous on the path space. Using Proposition \ref{skorokhod} we infer that 

\[
\mathscr{M}( {\varrho}_{n},{\vec m}_{n},{\vec U}_{n}, {C}_{n},{P}_{n}) \sim^d \mathscr{M}(\tilde{\varrho}_{n},\tilde{\vec m}_{n},\tilde{\vec U}_{n}, \tilde{C}_{n},\tilde{P}_{n}),
\]
where $\bfU_n:=\mathbb S_n(\nabla\bfu^n)$.
Fixing times $s,t \in [0,T]$, with $s<t$ we consider a continuous function $h$ such that
\[
h:\mathfrak X|_{[0,s]} \to [0,1].
\]
The process
\begin{align*}
 \mathscr{M}( {\varrho}_{n},{\vec m}_{n},{\vec U}_{n}, {C}_{n},{P}_{n})&=-\int_{\TN}\int_{0}^{t}\bfm_{n}\otimes\mathbb Q:\nabla\boldsymbol{\varphi} \,\mathrm{d}{W}^{n}\dx\\
 &= -\sum_{k=1}^K\int_{\TN}\int_{0}^{t}{\bfm}_{n}\otimes\bfsigma_k:\nabla\boldsymbol{\varphi} \,\mathrm{d}{W}_k^{n}\dx,
\end{align*}
is a square integrable $(\FF_t)$-martingale, consequently, we infer

\[
[\mathscr{M}( {\varrho}_{n},{\vec m}_{n},{\vec U}_{n}, {\mathrm C}_{n},{P}_{n})]^2 - \Psi(\bfm^n),\]
\[ \mathscr{M}( {\varrho}_{n},{\vec m}_{n},{\vec U}_{n}, {\mathrm C}_{n},{P}_{n})\beta_k -  \Psi_k(\bfm^n),
\]
are $(\FF_t)$-martingales. Now we set
\[
{\vec X}_{n}: = [{\varrho}_{n},{\vec m}_{n},{\vec U}_{n}, {\mathrm C}_{n},{\mathrm P}_{n}], \quad \tilde{\vec X}_{n}: =[\tilde{\varrho}_{n},\tilde{\vec m}_{n},\tilde{\vec U}_{n}, \tilde{\mathrm C}_{n},\tilde{\mathrm P}_{n}],\quad 
\tilde{\vec X}:=[ \tilde{\varrho},\tilde{\vec m},0, \tilde{\mathrm C},\tilde{\mathrm P}].
\]
 Let $\vec{r}_s$ be a restriction  function  to the interval $[0,s]$.   In view of Proposition \ref{skorokhod} and the equality of laws we obtain:
 
 \begin{equation}\label{eq:hmA}
\tilde{\E}\bigg[h(\vec{r}_s\tilde{\vec{X}}_{n},\vec{r}_s\tilde{W}^{n})\mathscr{M}(\tilde{\vec{X}}_{n})_{s,t}\bigg]={\E}\bigg[h(\vec{r}_s{\vec{X}}_{n},\vec{r}_s{W}^{n})\mathscr{M}({\vec{X}}_{n})_{s,t}\bigg]=0
\end{equation}

\begin{equation}\label{eq:hmB}
\tilde{\E}\bigg[h(\vec{r}_s\tilde{\vec{X}}_{n},\vec{r}_s\tilde{W}^{n})([\mathscr{M}(\tilde{\vec{X}}_{n})]^2 - \Psi(\tilde\bfm^n))_{s,t}\bigg]={\E}\bigg[h(\vec{r}_s{\vec{X}}_{n},\vec{r}_s{W}^{n})([\mathscr{M}({\vec{X}}_{n})]^2 - \Psi(\bfm^n))_{s,t}\bigg]=0
\end{equation}

\begin{equation}\label{eq:hmC}
\tilde{\E}\bigg[h(\vec{r}_s\tilde{\vec{X}}_{n},\vec{r}_s\tilde{W}^{n})(\mathscr{M}(\tilde{\vec{X}}_{n})\tilde W^n_k -  (\Psi_k(\tilde\bfm^n)))_{s,t})={\E}\bigg[h(\vec{r}_s{\vec{X}}_{n},\vec{r}_s{W}^{n})(\mathscr{M}({\vec{X}}_{n})W^n_k -  (\Psi_k(\bfm^n)))_{s,t})\bigg]=0
\end{equation}

Therefore, (\ref{variations}) and (\ref{variations_1}) hold, and consequently, (\ref{cross_var}) follows. Thus the momentum equation
\begin{align*}
\int_{\T}\tilde{\vec{m}}_{n}  \cdot \boldsymbol{\varphi} \, \dd x  &=\int_{\T}{\vec{m}}_{0}\cdot \boldsymbol{\varphi} \, \dd x + \int_{0}^{t}\int_{\T}\nabla \boldsymbol{\varphi}\, \dd \tilde{C}_{n} \dd s\\
&-\int_{0}^{t}\int_{\T}  \tilde{\vec{U}}_{n}:\nabla\boldsymbol{\varphi} \, \dd x \dd s 
+\int_{0}^{t}\int_{\T}\mathrm{div} \boldsymbol{\varphi} \, \dd \tilde{P}_{n} \dd s\\&
-\int_{\TN}\int_{0}^{\cdot}\tilde{\bfm}_{n}\otimes\mathbb Q:\nabla\boldsymbol{\varphi} \,\mathrm{d}\tilde{W}^{n}\dx
+\sum_{k=1}^K\int_0^t\int_{\mt}\Tilde\bfm\cdot\Div(\bfsigma_k\otimes\bfsigma_k\nabla_{x}\boldsymbol{\boldsymbol{\varphi}})\dx \,\dd s
\end{align*}

holds $\tilde{\p}$-a.s in new probability space $(\Tilde{\Omega}, \Tilde{\FF},\Tilde{\p})$. As similar argument can be employed to prove that the continuity equation continues to hold on the  new probability space, i.e., we have 
\begin{align}\label{3.13}
\begin{aligned}
\int_{\mt}\tilde\varrho^n\psi\dx&=\int_{\mt}\varrho_0\psi\dx+
\int_0^t\int_{\mt}\tilde\bfm^n\cdot\nabla_{x}\psi\dxs\\&
-\int_{\mt}\int_0^t\tilde\varrho \mathbb Q\cdot\nabla_x \psi\,\dd \tilde W^n\dx+\frac{1}{2}\sum_{k=1}^K\int_0^t\int_{\mt}\tilde\varrho\,\Div(\bfsigma_k \otimes\bfsigma_k\nabla_{x}\psi)\dx \ds\end{aligned}
\end{align}
$\mathbb P$-a.s.  for all $\psi\in C^\infty(\TN)$. Finally, the energy inequality only contains deterministic terms being measurable on the path space. Hence we have
\begin{align*}
\int_{\TN}\Big(\frac{1}{2} \frac{| {\tilde\bfm^n}(t) |^2}{\tilde\varrho^n(t)} + \frac{a}{\gamma-1}(\tilde\varrho^n(t))^\gamma\Big)\dx\leq \mathscr E_0
\end{align*}
$\mathbb P$-a.s. by Proposition \ref{skorokhod}.

\subsection{Passage to the limit}

To identify the limits in the nonlinear terms we first introduce defect measures. For this,  we adopt the notion of measures as presented in \cite{DBrt}.  In view of Proposition \ref{skorokhod}  we have
\[
p(\tilde{\varrho}_n)\to \overline{p(\tilde{\varrho})}\,\,\text{weakly-(*) in} \,\, L^{\infty}(0,T;\mathcal{M}^{+}(\T)) .
\]
 Noting that $p(\tilde{\varrho}) = a\tilde{\varrho}^{\gamma}$ is a convex function, we deduce
\[
0\leq p(\tilde{\varrho}) \leq \overline{p(\tilde{\varrho})},\quad \tilde{\mathcal{R}}_{\tt{press}}:= \overline{p(\tilde{\varrho})}- p(\tilde{\varrho}) \in L^{\infty}(0,T;\mathcal{M}^{+}(\T,R)) .
\]

Arguing similarly for the convective term,
\[
  \frac{\tilde{\Vec{m}}^n \otimes \tilde{\vec{m}}^n}{\tilde{\varrho}^n}\to \overline{\frac{\tilde{\Vec{m}} \otimes \tilde{\vec{m}}}{\tilde{\varrho}}} \,\,\text{weakly-(*) in} \,\, L^{\infty}(0,T;\mathcal{M}^{+}(\T,R^{N\times N})).
 \]
Setting
 \[
 \tilde{\mathcal{R}}_{\tt{conv}} := \overline{\frac{\tilde{\Vec{m}} \otimes \tilde{\vec{m}}}{\tilde{\varrho}}} - \frac{\tilde{\Vec{m}} \otimes \tilde{\vec{m}}}{\tilde{\varrho}},
 \]
 for $\bfxi \in \R^N$, convexity implies  
 \begin{eqnarray*}
 \tilde{\mathcal{R}}_{\tt{conv}}:(\bfxi \otimes \bfxi)&=&\lim_{n\to 0}\left[\frac{\tilde{\vec m}^{n}\otimes \tilde{\vec m}^{n}}{\tilde\varrho^{n}}:(\bfxi \otimes \bfxi)\right]- \frac{\tilde{\Vec{m}} \otimes \tilde{\vec{m}}}{\varrho}:(\bfxi \otimes \bfxi)\\
 &=&\lim_{n \to 0}\left[\frac{|\tilde{\vec m}^{n}\cdot \bfxi|^2}{\tilde{\varrho}^{n}} -\frac{|\tilde{\vec m}\cdot \bfxi|^2}{\tilde{\varrho}} \right] \geq 0,
 \end{eqnarray*}
 so that $\tilde{\mathcal{R}}_{\tt{conv}} \in L^{\infty}(0,T;\mathcal{M}^{+}(\T,\R^{N\times N}))$. 
To perform the limit in the stochastic term we use Lemma 2.1 in \cite{Debussche}. On account of the convergences in Proposition \ref{skorokhod}, Lemma 2.1 in \cite{Debussche} and the higher moments from (\ref{eq:hmA})-(\ref{eq:hmC})  we can pass to the limit $n\to \infty$ in the momentum equation \eqref{3.13} and obtain for all $t\geq0$

\begin{align}
\label{m_system}
\begin{aligned}
         \left[\int_{\T}\tilde{\vec m} \cdot \boldsymbol{\varphi}\dx\right]_{s=0}^{s=t}&=\int_{0}^{t}\int_{\T}\Big(\overline{\frac{\Tilde{\vec m} \otimes\Tilde{\vec m}}{\varrho}}:\nabla\boldsymbol{\varphi}+\overline{p(\tilde\varrho)}\,\Div\boldsymbol{\varphi}\Big)\, \dd x\, \dd s\\
         &-\int_{\mt}\int_0^t\Tilde\bfm\otimes\mathbb Q:\nabla_x \boldsymbol{\varphi}\,\dd\tilde W\dx\\&+\sum_{k=1}^K\int_0^\tau\int_{\mt}\tilde\bfm\cdot\Div(\bfsigma_k\otimes\bfsigma_k\nabla_{x}\boldsymbol{\varphi})\dx \ds
\end{aligned}
\end{align}
$\tilde{\mathbb P}$-a.s. for all $\boldsymbol{\varphi}\in C^\infty(\TN)$.
Consequently, the momentum equation in the sense of (\ref{eq:mcxs}) follows from rewriting (\ref{m_system}) using the defect measures introduced above. Since the continuity equation (\ref{eq:cont}) is linear, the limit passage is straightforward (as for the stochastic integral one can argue again by \cite[Lemma 2.1]{Debussche}).
Finally, since the energy inequality does not contain a stochastic integral, the limit follows directly from Proposition \ref{skorokhod} as in the deterministic case.

\section{Markov selection}
The aim of this section is to construct a Markov selection to \eqref{1.1}. In the following subsection we collect some basic material. Eventually, we present an alternative formulation for the martingale problem in Section \ref{markov}. The Markov selection theorem and its proof can be found in the last subsection.
\subsection{Basics}
Let $(X,\tau)$ be a topological space. We denote by $\BBB(X)$ the $\sigma$-algebra of Borel subsets of $(X,\tau)$. Let $\PP$ be a Borel measure on $X$, the symbol $\overline{\BBB(X)}$ denotes the $\sigma$-algebra of all Borel subsets of $X$ augmented by all zero measure sets. Let $\mathrm{Prob}[X]$ denote the set of all Borel probability measures on $X$. Furthermore, let $([0,1],\overline{\BBB[0,1]}, \LL)$ denote the standard probability space, where $\LL$ is the Lebesgue measure.

In regards to the notion of solutions in this paper, let $(X,d_{X})$ be a Polish space. For $t>0$ we introduce the path spaces
\[
\Omega_{X}^{[0,T]} =C([0,T];X),\qquad\Omega_{X}^{[T,\infty)} =C([T,\infty);X),\qquad\Omega_{X}^{[0,\infty)} =C([0,\infty);X),
\]
where the path spaces are Polish as long as $X$ is Polish, and we denote by $\BBB_T=\BBB(\Omega_{X}^{[0,T]})$ the Borelian $\sigma$-algebra on $\Omega_{X}^{[0,T]}$. Then, for $\omega \in \Omega_{X}^{[0,T]}$ we define a time shift operator 
\[
\Phi_{\tau}:\Omega_{X}^{[T,\infty)}\to \Omega_{X}^{[T+\tau,\infty)},\qquad \Phi_{\tau}[\omega]_s=\omega_{s-\tau},\, s\geq T+\tau,
\]
where $\Phi_{\tau}$ is an isometric mapping from $\Omega_{X}^{[t,\infty)}$ to $ \Omega_{X}^{[T+\tau,\infty)}$. For a Borel measure $\nu$ on $\Omega_{X}^{[T,\infty)}$, the time shift $\Phi_{-\tau}[\nu]$ is a Borel measure on the space $\Omega_{X}^{[T-\tau,\infty)}$ given by

\[
\Phi_{-\tau}[\nu](B)=\nu(\Phi_{\tau}(B)), \qquad B\in \BBB (\Omega_{X}^{[T-\tau,\infty)}).
\]
Since we aim at a strong Markov selection
we need a suitable filtration. For that purpose let
\[
\xi:\Omega_X^{[0,\infty)}\to \Omega_X^{[0,\infty)}, \quad (\xi(\omega))_t=\omega_t \in X,\,\,\text{for any $t\geq 0$},
\]
be the conical process and
\[
\B_t:= \sigma(\xi|_{[0,t]}),\quad t\geq 0,
\]
such that
$(\B_t)_{t\geq 0}$ is the associated filtration. Note that it
coincides with the Borel $\sigma$-field $\BBB_t$ on $\Omega^{[0,t]}_X=C([0,t];X)$. 
In the following parts, we recall important results of Stroock and Varadhan
\cite{STVAR} also used in \cite[Section 6]{HZZ}. Firstly, from Theorem 1.1.6 in \cite{STVAR} we obtain disintegration results, that is, existence of regular conditional probability laws.
\begin{thm}[Disintegration]\label{Dis}
Let $(X,d_X)$ be a polish space. Given $\PP\in\mathrm{Prob}[\Omega_{X}^{[0,\infty)}]$, let $\mathfrak t$ be a finite $(\B_t)$-stopping time. Then there exists a unique family of probability measures
\[
\PP|_{\B_{\mathfrak t}}^{\tilde{\omega}}\in \, \mathrm{Prob}[\Omega_{X}^{[0,\infty)}]\,\, \text{for}\,\, \PP\text{-a.a. }\tilde{\omega}
\]
 such that the mapping 
 \[
 \Omega_{X}^{[0,\infty)} \ni \tilde{\omega}\mapsto \PP|_{\B_{\mathfrak t}}^{\tilde{\omega}}\in \, \mathrm{Prob}[\Omega_{X}^{[0,\infty)}]
 \]
 is $\B_{\mathfrak t}$-measurable and the following properties hold:
 \begin{itemize}
     \item [(a)]  There is a $\mathcal P$-nullset
$\mathcal N$ such that for any $\omega\notin\mathcal N$
\begin{align*}
\mathcal P\Big(\Big\{\tilde\omega:\,\mathfrak t(\omega)=\mathfrak t(\tilde\omega),\,\xi(\omega)|_{[0,\mathfrak t(\tilde\omega)]}=\xi(\tilde\omega)|_{[0,\mathfrak t(\tilde\omega)]}\Big\}\Big)=1;
\end{align*}
     \item[(b)] For any Borel set $A\subset\mathscr B_{\mathfrak t}$ and any Borel set $B\subset\Omega_{X}^{[0,\infty)}$,
     \[
     \PP(\omega|_{[0,\mathfrak t]}\in A,\omega|_{[\mathfrak t,\infty)}\in B)= \int_{\tilde{\omega}\in A}\PP|_{\B_{\mathfrak t}}^{\tilde{\omega}}(B)\,\dd \PP(\tilde{\omega}).
     \]
 \end{itemize}
\end{thm}
Note that a conditional probability corresponds to disintegration of a probability measure with respect to a $\sigma$-field. Accordingly, \textit{reconstruction} can be understood as the opposite process of disintegration, that is, some sort of ``gluing together" procedure. Based on \cite[Theorem 6.1.2]{STVAR} we have the following result on reconstruction.

\begin{thm}[Reconstruction]\label{Rec}
Let $(X,d_X)$ be a Polish space. Let $\PP\in\mathrm{Prob}[\Omega_{X}^{[0,\infty)}]$ and $\mathfrak t$ be a finite $(\B_t)$-stopping time. Suppose that $Q_{\omega}$ is a family of probability measures, such that
\[
\Omega_{X}^{[0,\infty)} \ni \omega \mapsto Q_{\omega} \in\mathrm{Prob}[\Omega_{X}^{[T,\infty)}],
\]
is $\B_{\mathfrak t}$-measurable and it holds
\begin{align*}
Q_\omega\Big(\tilde\omega\in \Omega_{X}^{[0,\infty)}:\,\xi_{\mathfrak t(\omega)}(\tilde\omega)=\xi_{\mathfrak t(\omega)}(\omega)\Big)=1
\end{align*}
for any $\omega\in \Omega_{X}^{[0,\infty)}$.

 Then there exists a unique probability measure $\PP\otimes_{\mathfrak t}Q\in \mathrm{Prob}[\Omega_{X}^{[0,\infty)}]$ such that:
\begin{itemize}
    \item [(a)] For any Borel set $A \in \mathscr B_{\mathfrak t}$ we have 
    \[
    (\PP\otimes_{\mathfrak t}Q)(A)=\PP(A);
    \]
    \item[(b)] For $\PP\otimes_{\mathfrak t}Q$-a.a. $\omega\in \Omega_X^{[0,\infty)}$ we have
    \[
    (\delta_\omega\otimes_{\mathfrak t(\omega)}Q_\omega)=\PP\otimes_{\mathfrak t}Q|_{\B_{\mathfrak t}}^{{\omega}}.
    \]
\end{itemize}
\end{thm}
We follow the abstract framework on Markov processes along the lines of \cite{BFHAAP} and the references therein. We assume that $(X, d_{X})$ and $(F,d_F)$ are two Polish spaces such that the embedding $F\hookrightarrow X$ be continuous and dense. Moreover, let $Y$ be a Borel subset of $F$. Since $(Y,d_F)$ is not necessarily a complete \textit{space}, the embedding $Y\hookrightarrow X$ is not necessarily dense. Next we define probability measures with support only on certain subsets of a Polish space.
\begin{definition}
Let $(X, d_{X})$ and $(F,d_F)$ be two Polish spaces.
Let $Y$ be a Borel subset of $F$ and let $\PP\in\mathrm{Prob}[\Omega_{X}^{[0,\infty)}]$. A family of probability measures $\PP$ is concentrated  on the paths with values in $Y$ if there is some $A \in \B(\Omega_{X}^{[0,\infty)})$ such that $\PP(A)=1$ and $A\subset\{\omega\in \Omega_{X}^{[0,\infty)}:\omega(\tau)\in Y\,\forall \tau \geq 0\}$. We write $\PP\in$ $\mathrm{Prob}_{Y}[\Omega_{X}^{[0,\infty)}]$.
\end{definition}
We are now in position to define the strong Markov property suitable for our purposes.
\begin{definition}[Strong Markov property]\label{almsure}
Let $(X, d_{X})$ and $(F,d_F)$ be two Polish spaces and $Y$ a Borel subset of $F$.
Let $y\mapsto \PP_y$ be a measurable map defined on a measurable subset $Y\subset F$ with values in $\mathrm{Prob}_{Y}[\Omega_{X}^{[0,\infty)}]$. The family $\{\PP_y\}_{y\in Y}$ has the strong Markov property if for every $(\mathfrak B_t)$-stopping time $\mathfrak t$ and each $y\in Y$ it holds
\[
\PP_{\omega(\mathfrak t)}\circ \Phi_{\mathfrak t(\omega)}^{-1}=\PP_{y}|_{\B_{\mathfrak t}}^{\omega}\quad\text{for}\,\,\PP_{y}-\text{a.a. } \omega \in \Omega_{X}^{[0,\infty)}.
\]
\end{definition}
Finally, based on the link between disintegration and reconstruction as observed in \cite{Kr}, we have the following definition (see \cite[Theorem 12.2.3]{STVAR} and \cite{GoRo} for its generalisation to Polish spaces).
\begin{definition}[Pre-Markov family]\label{almostMark}
Let $(X, d_{X})$ and $(F,d_F)$ be two Polish spaces and
$Y$ a Borel subset of $F$. Let $\mathcal{C}:Y\to \mathrm{Comp}(\mathrm{Prob}[\Omega_{X}^{[0,\infty)}])\cap\mathrm{Prob}_{Y}[\Omega_{X}^{[0,\infty)}]$ be a measurable map, where $\mathrm{Comp}(\cdot)$ denotes the family of all compact subsets. The family $\{\mathcal{C}(y)\}_{y\in Y}$ is called pre-Markovian if for each $y\in Y$, every $\PP\in \mathcal{C}(y)$ 
and every finite $(\mathfrak B_t)$-stopping time $\mathfrak t$ we have:
\begin{itemize}
    \item [(a)] The disintegration property holds, that is, 
there is a $\mathcal P$-nullset
$\mathcal N$ such that for any $\omega\notin\mathcal N$
\begin{align*}
\xi_{\mathfrak t(\omega)}(\omega)\in Y,\quad\mathcal P(\Phi_{\mathfrak t(\omega)}(\cdot)|_{\mathscr B_{\mathfrak t}}^\omega)\in \mathcal C(\xi_{\mathfrak t(\omega)}(\omega));
\end{align*}
    \item[(b)] The reconstruction property holds, that is, if a map $\Omega_{X}^{[0,\infty)}\ni\omega\to Q_\omega\in \mathrm{Prob}(\Omega_{X}^{[\mathfrak t,\infty)})$ satisfies the assumptions of Theorem \ref{Rec} and there is a $\mathcal P$-nullset
$\mathcal N$ such that for any $\omega\notin\mathcal N$ it holds
    \[
\xi_{\mathfrak t(\omega)}(\omega)\in Y,\quad Q_\omega\circ\Phi_{\mathfrak t(\omega)}\in \mathcal C(\xi_{\mathfrak t(\omega)}(\omega)),
    \]
    then we have $\PP\otimes_{\mathfrak t}Q\in \mathcal{C}(y)$.
\end{itemize}
\end{definition}
Note that Definition \ref{almostMark} is motivated by results in \cite{FlaRom,GoRo} and was used similarly in \cite[Section 6]{HZZ}. We conclude the abstract framework by stating the following result which is a minor modification of \cite[Theorem 2.7]{GoRo} and can be proved similarly to \cite[Theorem 2.6]{BFHAAP}.\begin{thm}[Strong Markov selection]\label{Mthm} Let $(X, d_{X})$ and $(F,d_F)$ be two Polish spaces and $Y$ a Borel subset of $F$. Let $\{\mathcal{C}(y)\}_{y\in Y}$ be a pre-Markov family (as defined in Definition \ref{almostMark}) with nonempty convex values. Then there is a measurable map $y\mapsto \PP_y$ defined on $Y$ with values in $\mathrm{Prob}_{Y}[\Omega_{X}^{[0,\infty)}]$ such that $\PP_{y}\in \mathcal{C}(y)$ for all $y\in Y$ and $\{\PP_{y}\}_{y\in Y}$ has the strong Markov property (as defined in Definition \ref{almsure})
\end{thm}


\subsection{Martingale solutions as measures on the space of trajectories}\label{markov}
Firstly, we observe that from the proof of Theorem \ref{thm:mainEuler}, the natural filtration associated to a dissipative measure-valued martingale solution in the sense of Definition \ref{E:dfn} is the joint canonical filtration of $[\varrho, \vec m, \mathcal{R}_{\tt{conv}},\mathcal{R}_{\tt{press}},W]$. However, the processes $[\mathcal{R}_{\tt{conv}},\mathcal{R}_{\tt{press}}]$ are classes of equivalences in time and not stochastic processes in the classical sense. Therefore, it is not obvious as to how one should formulate the Markovianity of the system (\ref{1.1}). To circumvent this problem, we shall introduce the new variable $\vec R$ given by
\[
\vec R = \int_{0}^{\cdot}\left(\mathcal{R}_{\tt{conv}},\mathcal{R}_{\tt{press}}\right)\,\dd s.
\]
This idea goes back to \cite{BFHAAP} and has also been used in \cite{Mo}. 
Consequently, the notion of new variables allows us to establish the Markov selection for the joint law of $[\varrho,\vec m, \vec R]$. In this case, the stochastic process has continuous trajectories and contains all necessary information. The initial data for $\vec R$ is superfluous and only needed for technical reasons in the selection process.

To study Markov selection, it is desirable to consider the martingale solutions as probability measures $\PP \in $ Prob $[\Omega_{\tt{Euler}}]$ such that
\[
\Omega_{\tt{Euler}}= C_{\text{loc}}([0,\infty); W^{-k,2}(\T)),
\]
where $k>N/2$. Adopting the set-up of the previous subsection, we set $X= W^{-k,2}(\T)$. Accordingly, let $\B$ denote the Borel $\sigma$-field on $\Omega$. Let $\boldsymbol{\xi}=(\xi^1,\boldsymbol{\xi}^2,\boldsymbol{\xi}^3)$ denote the canonical processes of projections such that
\[
\boldsymbol{\xi}=(\xi^1,\boldsymbol{\xi}^2,\boldsymbol{\xi}^3):\Omega\to \Omega, \quad \boldsymbol{\xi}_t\omega=(\xi_t^1,\boldsymbol{\xi}_t^2,\boldsymbol{\xi}_t^3)(\omega)=\omega_t \in W^{-k,2}(\T),\text{for any $t\geq 0$},
\]
where the notation $\omega_t$ indicates that our random variable is time dependent. In addition, let $(\B_t)_{t\geq 0}$ be the filtration associated to canonical processes given by
\[
\B_t:= \sigma(\boldsymbol{\xi}|_{[0,t]}),\quad t\geq 0,
\]
which coincides with the Borel $\sigma$-field on $\Omega^{[0,t]}=([0,t];W^{-k,2}(\T))$. For the subsequent analysis, a dissipative martingale solution
\[
((\Omega,\FF, (\FF_t)_{t\geq 0},\p),\varrho,\vec m,\mathcal{R}_{\tt{conv}},\mathcal{R}_{\tt{press}}, W),
\]
in the sense of Definition \ref{E:dfn} will be considered as a probability law $\PP$, that is,
\[
\PP =\LL\left[\varrho, \vec m,\int_{0}^{\cdot} (\mathcal{R}_{\tt{conv}},\mathcal{R}_{\tt{press}})\, \dd s\right]\in\text{ Prob$[\Omega_{\tt{Euler}}]$}.
\]
Consequently, we obtain a filtered probability space $(\Omega,\B, (\B_t)_{t\geq 0}, \PP)$. Furthermore, we introduce the space
 \begin{eqnarray*}
 F &=& \left\{[\varrho,\vec m,\vec{R}] \in\Tilde{F}\bigg|\int_{\T}\frac{|\vec m|^2}{|\varrho|}\dd x<\infty\right\},\\
 \Tilde{F}&=&L^{\gamma}(\T)\times L^{\frac{2\gamma}{\gamma+1}}(\T)\times (W^{-k,2}(\T, R^{N^2+1})).
 \end{eqnarray*}
We augment $F$ with the points of the form $(0,\vec0,\vec R)$ for  $\vec R \in W^{-k,2}(\T,R^{N^2+1})$. Therefore, $F$ is a Polish space with metric
 \begin{equation}\label{metric}
    \dd_{F}(y,z)=\dd_{Y}((y^1,\vec y^2, \vec y^3),(z^1,\vec z^2,\vec z^3))=\|y-z\|_{\tilde F}+\left\|\frac{\vec y^2}{\sqrt{|y^1|}}-\frac{\vec z^2}{\sqrt{|z^1|}}\right\|_{L_x^2}.
 \end{equation}
 \newline \\
Moreover, the inclusion $F \hookrightarrow X$ is dense. Accordingly, we define a subset
 \[
 Y =\left\{[\varrho, \vec m,\vec{R}]\in X\bigg|\varrho\not\equiv 0, \varrho \geq 0, \int_{\T}\frac{|\vec m|^2}{\varrho}\, \dd x< \infty\right\}.
 \]
 We observe that $(Y,d_{F})$ is not complete because $\varrho\not\equiv0$, and the inclusion $Y\hookrightarrow X$ is not dense since $\varrho \geq 0$. We now a define a solution to the martingale problem. 
 
 \begin{definition}[Dissipative measure-valued martingale solution]\label{dfnMarkov} A Borel probability measure $\PP$ on $\Omega_{\tt{Euler}}$ is called a solution to the martingale problem associated to (\ref{1.1}) provided:
 \begin{itemize}
     \item [(a)] It holds
      \begin{align*}
          \PP&\left(\xi^1\in C_{\mathrm{loc}}[0,\infty);(L^{\gamma}(\T),w)),\xi^1\geq 0\right)=1,\\
           \PP&\left(\boldsymbol{\xi}^2\in C_{\mathrm{loc}}[0,\infty);(L^{\frac{2\gamma}{\gamma+1}}(\T),w))\right)=1,\\
             \PP&\left(\boldsymbol{\xi}^3=(\boldsymbol{\xi}^3_{\tt{conv}},{\xi}^3_{\tt{press}})\in W_{\mathrm{weak}-(*)}^{1,\infty}(0,\infty;\mathcal{M}^+(\T, \R^{3\times 3}\times\R))\right)=1;
     \end{align*}
     \item[(b)] 
The energy inequality is satisfied in the sense that
there is $z\geq \mathscr E_0$ such that
\begin{align} \label{eq:enedW2}
\begin{aligned}
- \int_0^\infty &\partial_t \psi \,
\mathfrak E_t \dt\leq
\psi(0) z
\end{aligned}
\end{align}
holds $\mathcal P$-a.s. for any $\psi \in C^\infty_c([0, \infty))$, where
     \begin{align*}
    \mathfrak{E}_t&:= \int_{\T}\left[\frac{1}{2}\frac{|\boldsymbol{\xi}^2|^2}{\xi^1}+\tfrac{a}{\gamma-1}(\xi^1)^{\gamma}\right](t)\, \dd x + \frac{1}{2}\int_{\T}\dd (\mathrm{tr}\partial_t\boldsymbol{\xi}^3_{\tt{conv}})(t) +\tfrac{1}{\gamma-1}\int_{\T}\dd(\partial_t{\xi}^3_{\tt{press}})(t),\\
\mathscr E_0&:=\int_{\mt}\left[\frac{1}{2}\frac{|\vec y^2|^2}{y^1}+\frac{a}{\gamma-1}(y^1)^\gamma\right]\dx;
    \end{align*}

\item[(c)] For any $\varphi \in C^1(\T)$,
the stochastic process

\begin{eqnarray*}
         \mathscr M^1(\varphi): [\omega, t] \mapsto\left[\int_{\T}{\xi}_t^1 \, {\varphi}\right]_{s=0}^{s=t}&+&\int_0^t \int_{\T}{\boldsymbol{\xi}_{t}^{2} \cdot \nabla \psi  } \,\dd x\ds\\
&& -\sum_{k=1}^K\int_0^t\int_{\mt}{\xi}^1_t\,\Div(\bfsigma_k\otimes\bfsigma_k\nabla_{x}{\varphi})\dx \ds\nonumber
    \end{eqnarray*}
is a square integrable $((\mathfrak{B}_{t})_{t\geq0},\PP)$-martingale with quadratic variation
\[
\frac{1}{2} \int_0^\cdot \sum_{k=1}^K \left( \int_{\T}{{\xi}_t^1\bfsigma_k\cdot\nabla_x {\varphi}}\, \dd x \right)^2\dd s;
\]
\item[(d)] For any $\boldsymbol{\varphi} \in C^1(\T, \R^3)$,
the stochastic process

\begin{eqnarray*}
         \mathscr M^2( \boldsymbol{\varphi}): [\omega, t] \mapsto\left[\int_{\T}\boldsymbol{\xi}_t^2 \cdot \boldsymbol{\varphi}\right]_{s=0}^{s=t}&-&\int_{0}^{t}\int_{\T}\left[\frac{\boldsymbol{\xi}_t^2 \otimes\boldsymbol{\xi}_t^2}{\xi_t^1}:\nabla\boldsymbol{\varphi}+p(\xi_t^1)\mathrm{div}\boldsymbol{\varphi}\right]\, \dd x \dd s\\
         &&-\int_{0}^{t}\int_{\TN}\nabla \boldsymbol{\varphi}:\dd (\partial_t\boldsymbol{\xi}_{\tt{conv}}^3) \dd s-\int_{0}^{t}\int_{\T}\mathrm{div} \boldsymbol{\varphi}\,\dd (\partial_t{\xi}_{\tt{press}}^3) \dd s\nonumber\\
&& -\sum_{k=1}^K\int_0^t\int_{\mt}\boldsymbol{\xi}^2_t\cdot\Div(\bfsigma_k\otimes\bfsigma_k\nabla_{x}\boldsymbol{\varphi})\dx \ds\nonumber
    \end{eqnarray*}
is a square integrable $((\mathfrak{B}_{t})_{t\geq0},\PP)$-martingale with quadratic variation
\[
\frac{1}{2} \int_0^\cdot \sum_{k=1}^K \left( \int_{\T}{\boldsymbol{\xi}_t^2\otimes\bfsigma_k:\nabla_x \boldsymbol{\varphi}}\, \dd x \right)^2\dd s
\]
and the corss variation between $ \mathscr M^2( \boldsymbol{\varphi})$ and  $\mathscr M^1( {\varphi})$ is given by
\[
\frac{1}{2} \int_0^\cdot \sum_{k=1}^K \left( \int_{\T}{\boldsymbol{\xi}_t^2\otimes\bfsigma_k:\nabla_x \boldsymbol{\varphi}}\, \dd x \right)\left( \int_{\T}{{\xi}_t^1\bfsigma_k\cdot\nabla_x {\varphi}}\, \dd x \right)\dd s.
\]
 \end{itemize}
 \end{definition}
 \begin{rmk}
The real number $z$ in Definition 
\ref{dfnMarkov} (b) allows of the possibility of
an initial energy jump.
It is reminiscent of \cite[Definition 3.3]{HZZ},
where $z$ is a stochastic process taking into account the stochastic terms in the energy inequality, which appear in the case of It\^{o}-type noise.
Since in our situation 
the energy balance is not effected by the noise,
it simplifies to a constant.
\end{rmk}
 In the following we state the relation between Definition \ref{E:dfn} and Definition \ref{dfnMarkov}.
 
 \begin{prop}\label{equdfns}
 The following statement holds true
 \begin{enumerate}
     \item Let $
((\Omega,\FF, (\FF_t)_{t\geq 0},\p),\varrho,\vec m,\mathcal{R}_{\tt{conv}},\mathcal{R}_{\tt{press}}, W)$ be a dissipative martingale solution in the sense of Definition \ref{E:dfn}. Then for every $\FF_0$-measurable random variables $\mathbf{R}_0$ with $\vec R_0 \in W^{-k,2}(\T,R^{N^2+1})$ we have that 
\begin{equation}\label{marklawA}
\PP= \LL\left[\varrho, \vec m,\vec R_0+\int_{0}^{\cdot}(\mathcal{R}_{\tt{conv}},\mathcal{R}_{\tt{press}})\,\dd s\right]\in \mathrm{Prob}[\Omega_{\tt{Euler}}]
\end{equation}
is a solution to the martingale problem in the sense of Definition \ref{dfnMarkov}.
\item Let $\PP$ be a solution to the martingale problem associated to (\ref{1.1}) in the sense of Definition \ref{dfnMarkov}, where $z= \mathscr E_0$. Then there exists a dissipative martingale solution $
((\Omega,\FF, (\FF_t)_{t\geq 0},\p),\varrho,\vec m, \mathcal{R}_{\tt{conv}},\mathcal{R}_{\tt{press}}, W)$  to the system (\ref{1.1}) in the sense of Definition \ref{E:dfn} 
and an $\FF_0$-measurable random variables $ \mathbf{R}_0$ with $\vec R_0 \in W^{-k,2}(\T,\R^{N^2+1})$ such that 
\begin{equation}\label{marklawB}
\PP= \LL\left[\varrho, \vec m,\vec R_0+\int_{0}^{\cdot}(\mathcal{R}_{\tt{conv}},\mathcal{R}_{\tt{press}})\,\dd s\right]\in \mathrm{Prob}[\Omega_{\tt{Euler}}].
\end{equation}
 \end{enumerate}
 \end{prop}
\begin{proof}
$(1)\Rightarrow (2)$:
Let $\big((\mathcal O,\mf,(\mf_t)_{t\geq0},\prst),\varrho,\bfm,\mathcal{R}_{\tt{conv}},\mathcal{R}_{\tt{press}},W)$ be a dissipative martingale solution in the sense of Definition \ref{E:dfn} and let $\mathbf R_0$ be an arbitrary
 $\mathfrak F_0$-measurable random variable with values in $W^{-k,2}(\mt,R^{N^2+1})$.
We shall show that the probability law given by \eqref{marklawA} is a solution to the martingale problem in the sense of Definition \ref{dfnMarkov}.

The point (a) in Definition \ref{dfnMarkov} follows  from Definition \ref{E:dfn}, \cite[Lemma 3.3]{BFHAAP} and the definition of $\PP$ as the pushforward measure generated by $[\vr,\bfm,\mathbf R]$.
Since the time integrals of the total energy are measurable functions on the subset of $\Omega_{\tt Euler}$, where the law $\PP$ is supported, we deduce that the points  (b) in Definition \ref{dfnMarkov} (with $z=\mathscr E_0$) hold. Next, we recall that by definition of the filtration $(\mathfrak{B}_{t})_{t\geq0}$, the canonical process $\bfxi=(\xi^{1},\bfxi^{2},\bfxi^{3})$ is $(\mathfrak{B}_{t})$-adapted. Hence by \cite[Lemma 3.3]{BFHAAP}, $\partial_{t}\bfxi^{3}$ is a $(\mathfrak{B}_{t})$-adapted random distribution taking values in $L^{\infty}(0,\infty;\mathcal M^{+}(\mt,R^{N\times N}\times R))$.

In order to show (c) and (d) we observe that all the  expressions appearing in the definition of $\mathscr{M}^1(\varphi)$ and $\mathscr M^2(\bfphi)$ are measurable functions on the subset of $\Omega_{\tt Euler}$ where $\PP$ is supported. Moreover, similarly to \cite[Lemma 3.3]{BFHAAP} we see that 
\begin{equation*}
\begin{split}
\mathfrak M^2(\bfphi)_t&:=\left[\int_{\T}\vec m \cdot \boldsymbol{\varphi}\dx\right]_{\tau=0}^{\tau=t}-\int_{0}^{t}\int_{\T}\Big(\frac{\vec m \otimes\vec m}{\varrho}:\nabla\boldsymbol{\varphi}+p(\varrho)\Big)\, \dd x\, \dd r\\
         &-\int_{0}^{t}\int_{\T}\nabla \boldsymbol{\varphi}:\dd \mathcal{R}_{\tt{conv}} \dd r-\int_{0}^{t}\int_{\T}\mathrm{div} \boldsymbol{\varphi}\,\dd \mathcal{R}_{\tt{press}} \dd r\\
         &-\sum_{k=1}^K\int_0^t\int_{\mt}\bfm\cdot\Div(\bfsigma_k\otimes\bfsigma_k\nabla_{x}\boldsymbol{\varphi})\dx \dd r.
\end{split}
\end{equation*}
is a martingale with respect to the canonical filtration generated by $[\vr,\bfm,\mathbf R]$. This directly implies the desired martingale property of $\mathscr{M}(\bfphi)$ as follows. We consider increments $X_{t,s}=X_t-X_s$, $s\leq t$, of stochastic processes. Then we obtain for $\bfphi\in C^\infty(\mt,R^N)$
and a continuous function $h:\Omega_{\tt Euler}^{[0,s]}\rightarrow[0,1]$ that
\begin{align*}
\E^{\mathcal P} \big[h(\bfxi|_{[0,s]})\mathscr M (\bfphi)_{s,t}\big]&=
\E^\p\big[h([\varrho,\bfm,\mathbf R]|_{[0,s]})\mathfrak M(\bfphi)_{s,t}\big]=0.
\end{align*}
Similarly, we obtain
\begin{align*}
&\E^U\bigg[h(\bfxi|_{[0,s]})\Big([\mathscr M^2 (\bfphi)^2]_{s,t}-\mathscr{N}^2(\bfphi)_{s,t}\Big)\bigg]=\E^\p\bigg[h([\varrho,\bfq,\bfU]|_{[0,s]})\Big([\mathfrak M^2 (\bfphi)^2]_{s,t}-\mathfrak N^2(\bfphi)_{s,t}\Big)\bigg]=0,
\end{align*}
where
\begin{align*}
\mathfrak N^2(\bfphi)_t&=\int_0^t\sum_{k=1}^K \left( \int_{\T}{\bfm\otimes\bfsigma_k:\nabla_x \boldsymbol{\varphi}}\, \dd x \right)^2\,\dif r,\\
\mathscr{N}^2(\bfphi)_t&= \int_0^t \sum_{k=1}^K \left( \int_{\T}{\boldsymbol{\xi}_t^2\otimes\bfsigma_k:\nabla_x \boldsymbol{\varphi}}\, \dd x \right)^2\,\dd r.
\end{align*}
As a consequence we deduce that $\mathscr M ^2(\bfphi)$ is a $(\mathfrak{B}_{t})$-martingale with quadratic variation $\mathscr{N}^2(\bfphi)$.

Similarly, we can prove that $\mathscr M ^1(\varphi)$ is a $(\mathfrak{B}_{t})$-martingale with quadratic variation 
$$\mathscr{N}^1(\varphi):=\int_0^t \sum_{k=1}^K \left( \int_{\T}{{\xi}_t^1\bfsigma_k\cdot\nabla {\varphi}}\, \dd x \right)^2\,\dd r.$$
Finally, we have
\begin{align*}
&\E^U\bigg[h(\bfxi|_{[0,s]})\Big([\mathscr M^2 (\bfphi)\mathscr M^1(\varphi)]_{s,t}-\mathscr{N}^{12}(\bfphi,\varphi)_{s,t}\Big)\bigg]\\&=\E^\p\bigg[h([\varrho,\bfq,\bfU]|_{[0,s]})\Big([\mathfrak M^2 (\bfphi)\mathfrak M^1(\varphi)]_{s,t}-\mathfrak N^{12}(\bfphi,\varphi)_{s,t}\Big)\bigg]=0,
\end{align*}
where
\begin{align*}
\mathfrak N^{12}(\bfphi,\varphi)_t&=\int_0^t\sum_{k=1}^K \left( \int_{\T}{\bfm\otimes\bfsigma_k:\nabla_x \boldsymbol{\varphi}}\, \dd x \right)\left( \int_{\T}{\varrho\bfsigma_k\cdot\nabla {\varphi}}\, \dd x \right)\,\dif r,\\
\mathscr{N}^2(\bfphi,\varphi)_t&= \int_0^t \sum_{k=1}^K \left( \int_{\T}{\boldsymbol{\xi}_t^2\otimes\bfsigma_k:\nabla_x \boldsymbol{\varphi}}\, \dd x \right)\left( \int_{\T}{{\xi}_t^1\bfsigma_k\cdot\nabla{\varphi}}\, \dd x \right)\,\dd r.
\end{align*}
This proves the formula for the cross variation.

$(2)\Rightarrow (1)$: Let $\PP\in{\rm Prob}[\OTN_{\tt Euler}]$ be a solution to the martingale problem in the sense of Definition \ref{dfnMarkov}.
We have to find a stochastic basis $(\mathcal O,\mf,(\mf_t)_{t\geq0},\prst)$, density $\vr$, momentum $\bfm$, defect measures $\mathcal{R}_{\tt{conv}}$ and $\mathcal{R}_{\tt{press}}$ as well as
a $K$-dimensional Wiener process $W$ such that $((\mathcal O,\mf,(\mf_t)_{t\geq0},\prst),\vr,\bfm, \mathcal{R}_{\tt{conv}},\mathcal{R}_{\tt{press}},W)$ is a dissipative martingale solution in the sense of Definition \ref{E:dfn}.

In view of (c) and (d) in Definition \ref{dfnMarkov} together with the standard martingale representation theorem (see \cite[Thm. 8.2]{Prato}) applied to $\mathscr M^1(\varphi)+\mathscr M^2(\bfphi)$ we infer the existence of an extended stochastic basis
\begin{equation*}
(\Omega\times\tilde\Omega,\mathfrak B\otimes{\tilde{\mathfrak B}},(\mathfrak B_{t}\otimes{\tilde{\mathfrak B}}_{t})_{t\geq0} ,U\otimes\tilde U),
\end{equation*}
and a $K$-dimensional Wiener process
$W$ adapted to $(\mathfrak B_{t}\otimes{\tilde{\mathfrak B}}_{t})_{t\geq0}$, such that
$$\mathscr M^1(\varphi)+\mathscr M^2(\bfphi)=-\sum_{k=1}^K\int_0^\tau\int_{\mt}\varrho\bfsigma_k:\nabla{\varphi}\,\dx\,\dd W_k -\sum_{k=1}^K\int_0^\tau\int_{\mt}\bfm\otimes\bfsigma_k:\nabla_x \boldsymbol{\varphi}\,\dx\,\dd W_k,$$
where
\[
\vr(\omega, \widetilde{\omega}) := \xi^1(\omega), \quad \bfm(\omega, \widetilde{\omega}) :=\bfxi^2(\omega),\quad(\mathcal R_{\tt{conv}},\mathcal R_{\tt{press}})(\omega,\tilde\omega):=\partial_t\bfxi^{3}(\omega).
\]
Choosing for $(\mathcal O,\mf,(\mf_t)_{t\geq0},\prst)$ the above extended probability space with the corresponding augmented filtration, then
$
 \big((\mathcal O,\mf,(\mf_t)_{t\geq0},\prst),\vr,\bfm,\mathcal R_{\tt{conv}},\mathcal R_{\tt{press}},W\big)
$
is a dissipative martingale solution solution
in the sense of Definition \ref{E:dfn}. Furthermore, it holds
\begin{equation*}
\PP =\mathcal{L} \left[ {\vr, \vc{m},\mathbf R=\mathbf R_{0}+\int_0^\cdot(\mathcal R_{\tt{conv}},\mathcal R_{\tt{press}})\,\dd s} \right] \in {\rm Prob}[\OTN_{\tt Euler}],
\end{equation*}
where by definition $\mathbf R_{0}(\omega,\tilde\omega)=\bfxi^{3}_{0}(\omega).$
\end{proof}

\subsection{Weak-strong uniqueness} Given a probability measure $\mathcal P\in\mathrm{Prob}[\Omega_{\tt Euler}]$ which is a solution to the martingale problem  in the sense of Definition \ref{dfnMarkov}
we construct a
 stochastic basis
\begin{equation*}
(\Omega\times\tilde\Omega,\mathfrak B\otimes{\tilde{\mathfrak B}},(\mathfrak B_{t}\otimes{\tilde{\mathfrak B}}_{t})_{t\geq0} ,\mathcal P\otimes\tilde{\mathcal P}),
\end{equation*}
and a $K$-dimensional Wiener process
$W$ adapted to $(\mathfrak B_{t}\otimes{\tilde{\mathfrak B}}_{t})_{t\geq0}$ as in the proof of Proposition \ref{equdfns}. By \cite{ART} there is a strong solution (in the sense of Definition \ref{def:compstrong}) $[r,\vec v]$ with stopping time $\mathfrak{s}$ defined on that probability space. We introduce the stopping time
\begin{align}\label{eq:0705}
\mathfrak s_R:=\inf\{t\in(0,\mathfrak s]:\,\|\vec v(t)\|_{C^1(\mt)}>R\}.
\end{align}
We obtain the following weak-strong uniqueness result which is in the spirit of \cite[Theorem 5.12]{FlaRom}.

\begin{thm}\label{thm_a} The weak-strong uniqueness holds true for the system (\ref{1.1A})--(\ref{1.1B}) in the following sense.
 Let 
\[((\Omega , \FF , (\FF)_{t \geq 0}, \mathbb{P} ),\varrho,\vec m, \mathcal{R}_{\tt{conv}},\mathcal{R}_{\tt{press}},W)
\]
be a dissipative martingale solution in the sense of Definition \ref{E:dfn}. Define the measure $\mathcal P$ in accordance with
\eqref{marklawA}, where $\vec R_0 \in W^{-k,2}(\T,\R^{15})$ is arbitrary. The we have for any $R\in\N$ 
\begin{align}\label{ew:wsulaw}
\mathbb E^{\mathcal P}[\mathbb I_{\{t\leq\mathfrak s_R\}}\eta(\bfxi_t)]=\mathbb E^{\mathcal P}[\mathbb I_{\{t\leq\mathfrak s_R\}}\eta(r(t),r\vec v(t),\mathbf R_0)]
\end{align}
for any $t\geq0$ and any $\eta:\Omega_{\tt Euler}\rightarrow\R$ measurable and bounded.
Here $[r,\vec v]$ and $\mathfrak{s}_R$ are given above in \eqref{eq:0705}.
\end{thm}
\begin{proof}
The proof is based on the relative energy
functional
\begin{align}\label{relative}
\begin{aligned}
\mathcal{K}\big({\bfxi}_t\big|r(t),\vec v(t)\big) &= \frac{1}{2}\int_{\T}\xi^1_t\Big|\frac{\bfxi^2_t}{\xi_t^1} -\vec v(t)\Big|^2\, \dd x+\int_{\T}\big(P(\xi^1)-P(r)+P'(r)(r-\xi^1))\,\dd x\\
&+ \frac{1}{2}\int_{\T}\dd (\mathrm{tr}\partial_t\boldsymbol{\xi}^3_{\tt{conv}})(t) +\tfrac{a}{\gamma-1}\int_{\T}\dd(\partial_t{\xi}^3_{\tt{press}})(t)
\end{aligned}
\end{align}
for $t\in[0,\mathfrak s]$, where $P(z)=\frac{a}{\gamma-1}z^\gamma$ denotes the pressure potential.

To compute $\dd \intTor{ \bfm \cdot \vc{v} }$ we apply a variant of the generalised It\^{o}-type formula from \cite[Theorem A.4.1]{BFHbook} obtaining
\begin{equation} \label{I1}
\begin{split}
\dd \left( \intTor{ \bfxi^2 \cdot \vec v } \right) &= \left( \intTor{ \left[ \bfxi^2 \cdot\left(-(\vec v\cdot\nabla)\vec v-\tfrac{1}{r}\nabla p(r)+ \nabla \vec v  : \frac{\bfxi^2\otimes\bfxi^2}{\xi^1} \right) + \Div \vec v\, p(\xi^1)  \right] } \right)  {\rm d}t \\
&-\int_{\T}\nabla \vec v:\dd (\partial_t\boldsymbol{\xi}_{\tt{conv}}^3) \dd t-\int_{\T}\mathrm{div} \vec v\,\dd (\partial_t{\xi}_{\tt{press}}^3) \dd t + \dd M_1  ,
\end{split}
\end{equation}
where
\[
M_1(t) = \sum_{k=1}^K\bigg(-\int_0^t \intTor{ \bfsigma_k\otimes\bfxi^2:\nabla\vc{v}} \,\circ\dd W_k + \int_0^t \intTor{ \bfxi^2 \cdot (\bfsigma_k\cdot\nabla) \vc{v} } \,\circ\dd W_k\bigg)=0.
\]
Similarly, we compute
\begin{equation} \label{I2}
\begin{split}
\dd \left( \intTor{ \frac{1}{2} \xi^1 |\vc{v}|^2 } \right) &=
\intTor{ \bfxi^2 \cdot \Grad {\bf v} \cdot {\bf v} }  \,{\rm d}t-  \intTor{ \xi^1 {\bf v} \cdot  ((\vec v\cdot\nabla)\vec v+\tfrac{1}{r}\nabla p(r)) } \,{\rm d}t
\\&+ {\rm d}M_2,
\end{split}
\end{equation}
\[
M_2 = \sum_{k=1}^K\bigg(\int_0^t \intTor{ \xi^1 \vc{v} \cdot (\bfsigma_k\cdot\nabla)\vec v } \,\circ \dd W_k-\frac{1}{2}\int_0^t \intTor{ \xi^1 \bfsigma_k\cdot\nabla|\vec v|^2 } \,\circ \dd W_k\bigg)=0,
\]
\begin{equation} \label{I3}
{\rm d} \left( \intTor{ \left[ P'(r) r - P(r) \right] } \right)
= -\intTor{ p'(r) \Div(r\vec v) } \ {\rm d}t  + \dd M_3,
\end{equation}
\[
M_3 = \sum_{k=1}^K\int_0^t \intTor{ r\bfsigma_k\cdot\nabla p'(r) } \,\circ\dd W_k=0,
\]
and, finally,
\begin{equation} \label{I4}
\begin{split}
{\rm d} \left( \intTor{ \xi^1 P'(r) } \right) &=
 \intTor{ \xi^1 \Grad P'(r) \cdot \vec v } \ {\rm d}t -
\intTor{ \xi^1 P''(r) \Div(r\vec v)} \ {\rm d}t 
+ \dd M_4,
\end{split}
\end{equation}
\[
M_4(t) = -\sum_{k=1}^K\int_0^t \intTor{ \xi^1 r\bfsigma_k\cdot\nabla P'(r) } \,\circ\dd W_k+\int_0^t \intTor{ \xi^1 P''(r) \nabla r\cdot\bfsigma_k } \,\circ\dd W_k=0.
\]
Finally, we collect and sum the resulting expressions \eqref{I1}--\eqref{I4}, and add the energy balance from Definition \ref{dfnMarkov} (b) to the sum obtaining
\begin{align}\label{REI}
\begin{aligned}
   \mathcal{K}\big(\bfxi\big|r,\vec v\big)&= \,\mathcal{K}\big(\xi^1,\bfxi^2,\bfxi^3\big|r,\vec v\big)(0)+\int_{0}^{\tau}\mathcal{Q}\big(\bfxi\big|r,\vec v\big)\, \dd t
\end{aligned}
\end{align}
$\p$-a.s for all $\tau \in (0,\mathfrak s)$, where
\begin{align*}
\mathcal{Q}\big(\bfxi\big|r,\vec v\big)=
&\int_{\T}\xi^1\left(\frac{\bfxi^2}{\xi^1}-\vec v\right)\cdot\nabla\vec v\cdot \left(\vec v-\frac{\bfxi^2}{\xi^1}\right)\,\dd x\nonumber\\
&+\int_{\T}\Big[-\frac{1}{r}\nabla p(r)\cdot(\xi^1\vec v) -p(\xi^1)\mathrm{div}\vec v\Big] \, \dd x  \\
&+\int_{\T} \left(\left(\frac{\xi^1}{r}-1\right)p'(r)\Div(r\vec v)\right)\,\dd x\nonumber\\
&-\int_{\T}\nabla \vec v:\dd (\partial_t\boldsymbol{\xi}_{\tt{conv}}^3) \dd t-\int_{\T}\mathrm{div} \vec v\,\dd (\partial_t{\xi}_{\tt{press}}^3) \dd t\\
&=\int_{\T}\xi^1\left(\frac{\bfxi^2}{\xi^1}-\vec v\right)\cdot\nabla\vec v\cdot \left(\vec v-\frac{\bfxi^2}{\xi^1}\right)\,\dd x\nonumber\\
&+\int_{\T}\big(P(\xi^1)-P(r)+P'(r)(r-\xi^1))\,\dd x\\
&-\int_{\T}\nabla \vec v:\dd (\partial_t\boldsymbol{\xi}_{\tt{conv}}^3) \dd t-\int_{\T}\mathrm{div} \vec v\,\dd (\partial_t{\xi}_{\tt{press}}^3) \dd t.
\end{align*}
Finally, we note that the standard maximum principle for the continuity equation still applies despite the transport noise. This yields that $r$ can be bounded from above and below in terms of $\|\nabla\vec v\|_{C^1(\mt)}$ and thus
$\mathcal{Q}\big(\bfxi\big|r,\vec v\big)\leq\,c(R)\mathcal{K}\big(\bfxi\big|r,\vec v\big)$ for $\tau \in (0,\mathfrak s_R)$.
Hence we obtain $\mathcal{K}\big(\bfxi\big|r,\vec v\big)=0$  for $\tau \in (0,\mathfrak s_R)$ from Gronwall's lemma.
\end{proof}

\subsection{The main theorem}
 In this section we state and prove the existence of a strong Markov selection to the complete stochastic Euler system (\ref{1.1}). Let $y\in Y$ be an admissible initial datum. We denote by $\PP_{y}$ a solution to the martingale problem associated with (\ref{1.1}) starting on $y$ at time $t=0$; that is, the marginal of $\PP_y$ at $t=0$ is $\delta_y$. 
%
 

 \begin{thm}\label{Mselection}
Let $\bfsigma_k \in C^\infty (\TN, R^N)$, ${\rm div}\bfsigma_k = 0$, $k=1,\dots,K$, and let $\gamma > \frac{N}{2}$, $N= 2,3$ be given. Then there exists a family $\{\PP_y\}_{y\in Y}$ of solutions to the martingale problem associated to (\ref{1.1}) in the sense of Definition \ref{dfnMarkov} with the strong Markov property (as defined in Definition \ref{almostMark})
 \end{thm}
 
 We set $y =(y^1,\vec y^2,\vec y^3) \in Y$ and denote by $\CC(y)$ the set of probability laws $\PP_{y}\in \mathrm{Prob}[\Omega_{{\tt Euler}}]$ solving the martingale problem associated to (\ref{1.1}) with initial condition $(y^1,\vec y^2)$. The proof of Theorem \ref{Mselection} follows from applying the abstract result of Theorem \ref{Mthm}. In particular, we show that the family $\{\CC(y)\}_{y\in Y}$ of solutions to the martingale problem satisfies the disintegration and reconstruction properties in Definition \ref{almostMark}.
 
 \begin{lem}
 For each $y =(y^1,\vec y^2,\vec y^3) \in Y$. The set $\CC(y)$ is \textit{non-empty} and \textit{convex}. Furthermore, for every $\PP\in \CC(y)$, the marginal at every time $t\in (0,\infty)$ is supported on $Y$.
 \end{lem}
 
 \begin{proof}
 Assuming $y\in Y$, an application of Theorem \ref{thm:main} yields existence of a martingale solution to the problem (\ref{1.1}) in the sense of Definition \ref{E:dfn}. Consequently, by Proposition \ref{equdfns} we infer that for each $y\in Y$ the set $\CC(y)$ is non-empty. For some $\lambda \in (0,1)$, let $\PP_1,\PP_2 \in \CC(y)$ be such that $\PP=\lambda \PP_1+(1-\lambda)\PP_2$. Then convexity follows from noting that the properties from Definition \ref{E:dfn} involve integration with respect to the elements of $\CC(y)$. In view of Definition \ref{E:dfn} property (f) (i.e., the energy inequality), the marginal $\PP\in \CC(y)$ at every $t\in (0,\infty)$ is supported in $Y$.
 \end{proof}
 
 For compactness we consider the following Lemma.
 \begin{lem}
 Let $y\in Y$. Then $\CC(y)$ is a compact set and the map $\CC:Y\to \mathrm{Comp}(\mathrm{Prob}[\Omega_{{\tt Euler}}])$ is Borel measurable.
 
 \end{lem}
 \begin{proof}
Measurability of the map $y\mapsto \CC(y)$ follows from using Theorem 12.1.8 in \cite{STVAR} for the metric space $(Y,d_F)$. 
 Let $(y_n = (\varrho_n,\vec m_n,\vec R_n))_{n\in \N} \subset Y$ be a sequence converging in $Y$ (to some $y = (\varrho,\vec m,\vec R))$ with respect to the metric $d_F$ in (\ref{metric}). Let $\PP_n \in C(y_n), n\in \N$.
 Compactness can be proved following the lines of the proof of Theorem \ref{thm:mainEuler}. That is, after taking a non-relabelled, $(\PP_n)_{n\in N}$ converges to some $\PP \in \CC(y)$ weakly in $\mathrm{Prob}[\Omega_{\tt Euler}]$. Consequently, we can argue as in the proof of Theorem \ref{thm:mainEuler} to show that $\PP$ is a solution to the martingale problem with initial law $\delta_y$. Therefore, $\PP\in \CC(y)$ as required.
 \end{proof}
 
 Finally, we verify that $\CC(y)$ has the disintegration and reconstruction property in sense of Definition  \ref{almostMark}.
 
 \begin{lem}\label{dis}
 The family $\{\CC(y)\}_{y\in Y}$ satisfies the disintegration property
 of Definition \ref{almostMark}.
 \end{lem}
 
 \begin{proof}
 Fix $y\in Y$, $\PP\in \CC(y)$ and let $\mathfrak t$ be $(\B_t)$-stopping time. In view of Theorem \ref{Dis}, we know there exists a family of probability measures;
 \[
 \Omega \ni \tilde{\omega}\mapsto \PP|_{\B_{\mathfrak t}}^{\tilde{\omega}}\in \mathrm{Prob}[\Omega_{\tt Euler}^{[0,\infty)}]
 \]
 such that 
there is a $\mathcal P$-nullset
$\mathcal N$ such that for any $\omega\notin\mathcal N$
\begin{align}\label{ameas}\begin{aligned}
\mathcal P\Big(\Big\{\tilde\omega:\,\mathfrak t(\omega)=\mathfrak t(\tilde\omega),\,\xi(\omega)|_{[0,\mathfrak t(\tilde\omega)]}=\xi(\tilde\omega)|_{[0,\mathfrak t(\tilde\omega)]}\Big\}\Big)=1,\\
     \PP(\omega|_{[0,\mathfrak t]}\in A,\omega|_{[\mathfrak t,\infty)}\in B)= \int_{\tilde{\omega}\in A}\PP|_{\B_{\mathfrak t}}^{\tilde{\omega}}(B)\,\dd \PP(\tilde{\omega}),
\end{aligned}
     \end{align}
for any Borel set $A\subset\mathscr B_{\mathfrak t}$ and any Borel set $B\subset\Omega_{X}^{[0,\infty)}$.
 Here, we want to show that 
there is a $\mathcal P$-nullset
$N$ such that for any $\omega\notin N$
\begin{align*}
\xi_{\mathfrak t(\omega)}(\omega)\in Y,\quad\mathcal P(\Phi_{\mathfrak t(\omega)}(\cdot)|_{\mathscr B_{\mathfrak t}}^\omega)\in \mathcal C(\xi_{\mathfrak t(\omega)}(\omega));
\end{align*}
  Thus we are seeking a $\PP$-nullset $N$ outside of which properties (a)-(d) of Definition \ref{dfnMarkov} hold for $\mathcal P(\Phi_{\mathfrak t(\omega)}(\cdot)|_{\mathscr B_{\mathfrak t}}^\omega)$, where the initial datum is $\xi_{\mathfrak t(\omega)}(\omega)$. To begin with, set $N_a,\ldots,N_d$ to each of the properties (a)-(d) of Definition \ref{dfnMarkov}, respectively, and let $N= N_a\cup\cdots\cup N_d$. Arguing along the lines of Lemma 4.4 in \cite{FlaRom} (see also \cite{BFHAAP} and \cite{Mo}) we have the following observations:
  \begin{itemize}
      \item[(1)] 
First of all, $\mathcal P(\Phi_{\mathfrak t(\omega)}(\cdot)|_{\mathscr B_{\mathfrak t}}^\omega)$ has the correct initial datum $\xi_{\mathfrak t(\omega)}(\omega)$ (outside a nullset) due to \eqref{ameas} above.
Setting
       \begin{align*}
           H_{\mathfrak t}&=\bigg\{\omega \in \Omega:\omega|_{[0,\mathfrak t(\omega)]} \in C([0,,\mathfrak t(\omega)];L^{\gamma}(\T))\times C([0,,\mathfrak t(\omega)];L^{\frac{2\gamma}{\gamma+1}}(\T))\times\\
          &\qquad\qquad\qquad\qquad\qquad\times W_{\text{weak-*}}^{1,\infty}(0,,\mathfrak t(\omega);\mathcal{M}^{+}(\T;R^{N^2+1}))\bigg\}\\
            H^{\mathfrak t}&=\bigg\{\omega \in \Omega:\omega|_{[\mathfrak t(\omega),\infty)} \in C([\mathfrak t(\omega),\infty);L^{\gamma}(\T))\times C([\mathfrak t(\omega),\infty);L^{\frac{2\gamma}{\gamma+1}}(\T))\times\\&\qquad\qquad\qquad\qquad\qquad\times W_{\text{weak-*}}^{1,\infty}(\mathfrak t(\omega),\infty;\mathcal{M}^{+}(\T;R^{N^2+1}))\bigg\},
      \end{align*}
      in view of property (a)  in Definition \ref{dfnMarkov} for $\PP$ we obtain
      \[1= \PP(H_{\mathfrak t}\cap H^{\mathfrak t}) = \int_{\tilde\omega \in H_{\mathfrak t}}\PP|_{\B_{\mathfrak t}}^{\tilde\omega}(H^{\mathfrak t})\,\dd \PP(\tilde{\omega})
.\]
      Therefore, there is a $\PP$-nullset $N_a$ (which belongs to $\mathfrak B_{\mathfrak t}$) such that $\PP|_{\B_T}^{\tilde\omega}(H^{\mathfrak t}) =1$ holds for $\tilde{\omega}\notin N_a$.
      \item[(2)] For property (b)
we split the energy balance as follows:
      \begin{gather}
 \mathfrak E_\tau\leq\mathfrak E_s=z\quad \text{for a.a.} \,\,0\leq s\leq  \tau\leq \mathfrak t,\label{ceqA}\\
 \mathfrak E_\tau\leq \mathfrak E_s=z\quad \text{for a.a.}\,\, \mathfrak t\leq s\leq  \tau<\infty,\label{ceqB}
\end{gather}
and consider the sets
\begin{gather*}
    \mathfrak{A}_{\mathfrak t}=\{\omega \in \Omega:\omega|_{[0,\mathfrak t(\omega)]}\,\text{satisfies (\ref{ceqA})} \}\\
    \mathfrak{A}^{\mathfrak t}=\{\omega \in \Omega:\omega|_{[\mathfrak t(\omega),\infty)}\,\text{satisfies (\ref{ceqB})} \}.
\end{gather*}
    As the property (b) holds for $\PP$, arguing similarly as in proof of (a) yields a nullset $N_b$.
    \item[(3)]
The existence of $\PP$-nullsets $N_c$ and $N_d$
such that $\mathcal P(\Phi_{\mathfrak t(\omega)}(\cdot)|_{\mathscr B_{\mathfrak t}}^\omega)$ satisfies (c) and (d) outside of $N_c$ respectively $N_d$ 
 follows  from \cite[Theorem 1.2.10]{STVAR} applied to the $R^{N+1}$-valued process
\begin{align}\label{eq:0610}
[\omega,t]\mapsto \begin{pmatrix}\mathscr M^1_t(\varphi_n)\\\mathscr M^2_t(\boldsymbol{\varphi}_n)
\end{pmatrix}
\end{align}
for dense subsets $(\varphi_n)\subset C^\infty(\mt)$ and $(\boldsymbol{\varphi}_n)\subset C^\infty(\mt,R^N)$, respectively. 
  \end{itemize}
 Choosing $N= N_a\cup\cdots\cup N_d$ completes the proof.
 \end{proof}
 
 \begin{lem}\label{ris}
 The family $\{\CC(y)\}_{y\in Y}$ satisfies the reconstruction property
 of Definition \ref{almostMark}.
 \end{lem}
 
 \begin{proof}
 Fix $y\in Y$, $\PP\in \CC(y)$ and let $\mathfrak t$ be $(\B_t)$-stopping {time}. In view of Theorem \ref{Rec},
suppose that $Q_{\omega}$ is a family of probability measures, such that
\[
\Omega_{{\tt Euler}}^{[0,\infty)} \ni \omega \mapsto Q_{\omega} \in\mathrm{Prob}[\Omega_{{\tt Euler}}^{[0,\infty)}]
\]
is $\B_{\mathfrak t}$-measurable and it holds
\begin{align*}
Q_\omega\Big(\tilde\omega\in \Omega_{{\tt Euler}}^{[0,\infty)}:\,\xi_{\mathfrak t(\omega)}(\tilde\omega)=\xi_{\mathfrak t(\omega)}(\omega)\Big)=1
\end{align*}
for any $\omega\in \Omega_{{\tt Euler}}^{[0,\infty)}$.
 Then there exists a unique probability measure $\PP\otimes_{\mathfrak t}Q\in \mathrm{Prob}[\Omega_{{\tt Euler}}^{[0,\infty)}]$ such that:
\begin{itemize}
    \item [(a)] For any Borel set $A \in \mathscr B_{\mathfrak t}$ we have 
    \[
    (\PP\otimes_{\mathfrak t}Q)(A)=\PP(A);
    \]
    \item[(b)] For $\PP\otimes_{\mathfrak t}Q$-a.a. $\omega\in \Omega_{{\tt Euler}}^{[0,\infty)}$ we have
    \[
    (\delta_\omega\otimes_{\mathfrak t(\omega)}Q_\omega)=\PP\otimes_{\mathfrak t}Q|_{\B_{\mathfrak t}}^{{\omega}}.
    \]
\end{itemize}
We aim to prove that if there is a $\mathcal P$-nullset
$N$ such that for any $\omega\notin N$ it holds
    \[
\xi_{\mathfrak t(\omega)}(\omega)\in Y,\quad Q_\omega\circ\Phi_{\mathfrak t(\omega)}\in \mathcal C(\xi_{\mathfrak t(\omega)}(\omega)),
    \]
    then we have $\PP\otimes_{\mathfrak t}Q\in \mathcal{C}(y)$.

 In order to do this we have to verify properties (a)-(d) in Definition \ref{dfnMarkov}. The proof follows along the lines of \cite{FlaRom}, Lemma 4.5. Adopting the notation introduced in Lemma \ref{dis}, we argue as follows:\\
Here we note that $Q_{\omega}$ is a regular conditional probability distribution of  $(\PP\otimes_{\mathfrak t}Q)$ with respect to $\B_{\mathfrak t}$.
\begin{itemize}
    \item [(1)] Since (a) holds for $Q_{\omega}$ we have $Q_{\omega}(H^{\mathfrak t})=1$ such that
    \[
    \PP\otimes_{{\mathfrak t}}Q(H_{\mathfrak t}\cap H^{\mathfrak t}) =\int_{H_{\mathfrak t}}Q_{\omega}[H^{\mathfrak t}]\,\dd \PP(\omega)=1.
    \]
    \item [(2)] For property (b) of Definition \ref{dfnMarkov} we argue as in property (a) (Using the notation developed in Lemma \ref{dis} for each property, respectively).
    \item[(3)] In the case of the properties (c) and (d) we can argue via the process from \eqref{eq:0610}.
It follows again by \cite[Lemma 1.2.10]{STVAR} that it is a
 $((\B_t)_{t\geq 0},\PP\otimes_{\mathfrak t}Q)$-martingale with corresponding quadratic variation.
\end{itemize}
 \end{proof}

\section*{Acknowledgement}

The work of D.B.
was supported by the German Research Foundation (DFG) within the framework of the priority research program SPP 2410 under the grant BR 4302/3-1 (525608987) and under the personal grant BR 4302/5-1 (543675748).

\section*{Conflict of Interest}
The authors have no conflict of interest to declare.

\end{document}